\documentclass[11pt,a4]{article}
\usepackage{graphicx}
\usepackage{color}

\usepackage{epsfig,amsthm}
\usepackage{amssymb,amsfonts}
\usepackage{dsfont}
\usepackage{graphicx}
\usepackage{subfigure}
\usepackage{indentfirst}
\usepackage{mathrsfs}
\usepackage{amsmath}
\usepackage{setspace}

\newcommand{\D}{\mathrm{d}}

\newtheorem{thm}{Theorem}
\newtheorem{cor}{Corollary}
\newtheorem{defi}{Definition}
\newtheorem{lem}{Lemma}
\newtheorem{prop}{Proposition}
\newtheorem{rmk}{Remark}

\title{Energy Scaling and Asymptotic Properties of One-Dimensional Discrete
System with Generalized Lennard--Jones $(m,n)$ Interaction\footnotemark[3]
}

\author{
    Tao Luo
    \footnotemark[1] \footnotemark[2]\and
    Yang Xiang
    \footnotemark[1]\and
    Nung Kwan Yip
    \footnotemark[2]
}

\footnotetext[1]{Department of Mathematics, Hong Kong University of Science and Technology, Clear Water Bay, Hong Kong.
(maxiang@ust.hk)}
\footnotetext[2]{Department of Mathematics, Purdue University, West Lafayette, IN 47907, USA.
(luo196@purdue.edu, yip@math.purdue.edu)}
\footnotetext[3]{Abbreviated title: Scaling laws of non-local Lennard--Jones systems}

\date{Received: date / Accepted: date}

\begin{document}
%\doublespacing
%\begin{spacing}{2.0}
\maketitle
\allowdisplaybreaks
\noindent
\begin{abstract}
It is well known that elastic effects can cause surface instability. 
In this paper, we analyze a one-dimensional discrete system which
can reveal pattern formation mechanism resembling the ``step-bunching'' 
phenomena for epitaxial growth on vicinal surfaces. 
The surface steps are subject to long range pairwise interactions
taking the form of a general Lennard--Jones (LJ) type potential. It is 
characterized by two exponents $m$ and $n$ describing the singular and
decaying behaviors of the interacting potential at small and large distances,
and henceforth are called generalized LJ $(m,n)$ potential. 
We provide a systematic analysis of the asymptotic properties of the step 
configurations and the value of the minimum energy, in particular
their dependence on $m$ and $n$ and an additional parameter $\alpha$ 
indicating the interaction range.
Our results show that there is a phase transition between the 
bunching and non-bunching regimes. Moreover, some of our statements 
are applicable for 
any critical points of the energy, not necessarily minimizers.
This work extends the technique and results of \cite{Luo2016-p737-771} 
which concentrates on the case of LJ (0,2) potential (originated
from the elastic force monopole and dipole interactions between the steps). 
As a by-product, our result also leads to the well known fact that
the classical LJ (6,12) potential does not demonstrate step-bunching type
phenomena.
\end{abstract}

\newpage
\doublespacing
\section{Introduction}
Elasticity effects, which may cause surface morphological instability, are widely believed to be important in epitaxial film growth. 
In particular, the elastic effects can lead to the so called 
{\em step-bunching} instability.
This phenomenon has been modelled via both discrete \cite{Tersoff1995-p2730-2733,Liu1998-p1268-1271,Duport1995-p134-137,Duport1995-p1317-1350} and continuum \cite{Xiang2002-p241-258,Xiang2004-p35409-35409} approaches. Linear stability analysis and numerical simulations of these models have shown excellent agreement 
with experiments on epitaxial growth on vicinal surfaces \cite{Tersoff1995-p2730-2733,Xiang2004-p35409-35409}. Recently, the work \cite{Luo2016-p737-771} rigorously 
demonstrates the presence of step bunching and characterizes its profile as 
well as some scaling laws for an associated elastic energy. 
The results show that the bunching phenomenon depends very much on the form of the underlying interaction between the steps. In this work, we extend the method in \cite{Luo2016-p737-771} to investigate the energy scaling laws and other asymptotic behaviors of one-dimensional system with general pairwise interactions.
The key technical difficulty is the {\em non-locality} of the interaction 
between the steps. We refer to \cite{politi2000instabilities} 
for a review of various surface instability mechanisms in epitaxial growth.

In \cite{Luo2016-p737-771}, we studied the model originally introduced by Tersoff et al. \cite{Tersoff1995-p2730-2733,Liu1998-p1268-1271} for epitaxial growth on vicinal surface with elastic effects between the steps. 
Another model proposed by Duport et al. 
\cite{Duport1995-p134-137,Duport1995-p1317-1350} incorporated elastic interaction between adatoms and steps, the Schwoebel barrier, and other kinetic effects. 
The Tersoff's model concentrates mostly on the energetic or relaxation 
phenomena.
It is a discrete atomistic model, tracking all the step positions 
$\{x_i: x_i < x_{i+1}\}_{i\in \mathbb{Z}}$, which evolve according to the following dynamical law,
\begin{equation}\label{eq..TersoffOriginal}
  \frac{\D x_i}{\D t}=F_{\mathrm{ad}}\frac{x_{i+1}-x_{i-1}}{2}+B\left(\frac{\mu_{i+1}-\mu_i}{x_{i+1}-x_i}-\frac{\mu_i-\mu_{i-1}}{x_i-x_{i-1}}\right),\, i\in\mathbb{Z}.
\end{equation}
In the above, $l_i := x_{i+1} - x_i$ is the length of the $i$-th terrace 
(cf. Fig. \ref{figure..step}(a)), 
\begin{figure}
  %\subfigure{\includegraphics[width=0.6\textwidth]{step.pdf}}
  \begin{center}
  \mbox{(a)\includegraphics[height=0.25\textwidth]{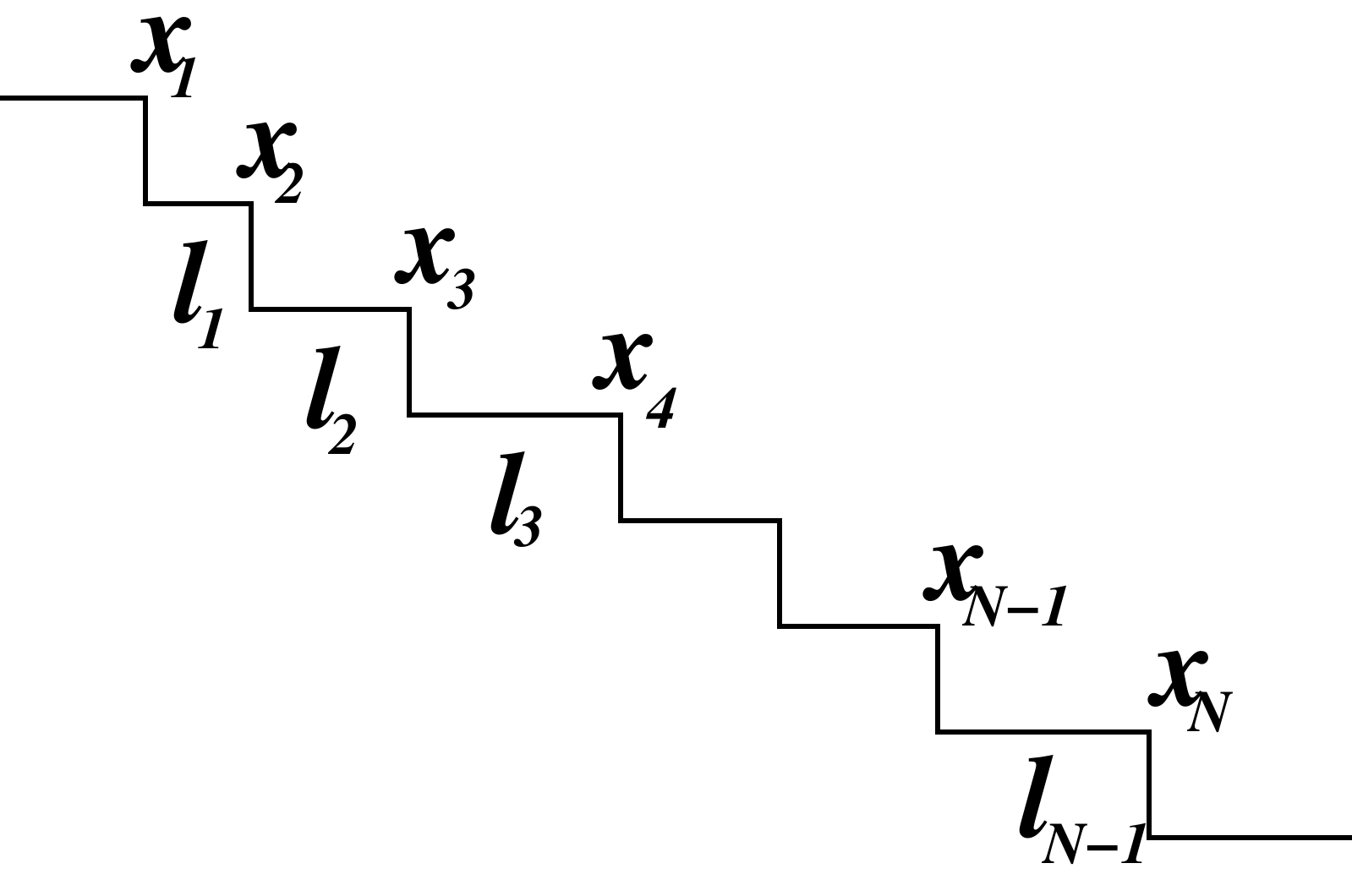}
\hspace{10pt}
  (b)\includegraphics[height=0.25\textwidth]{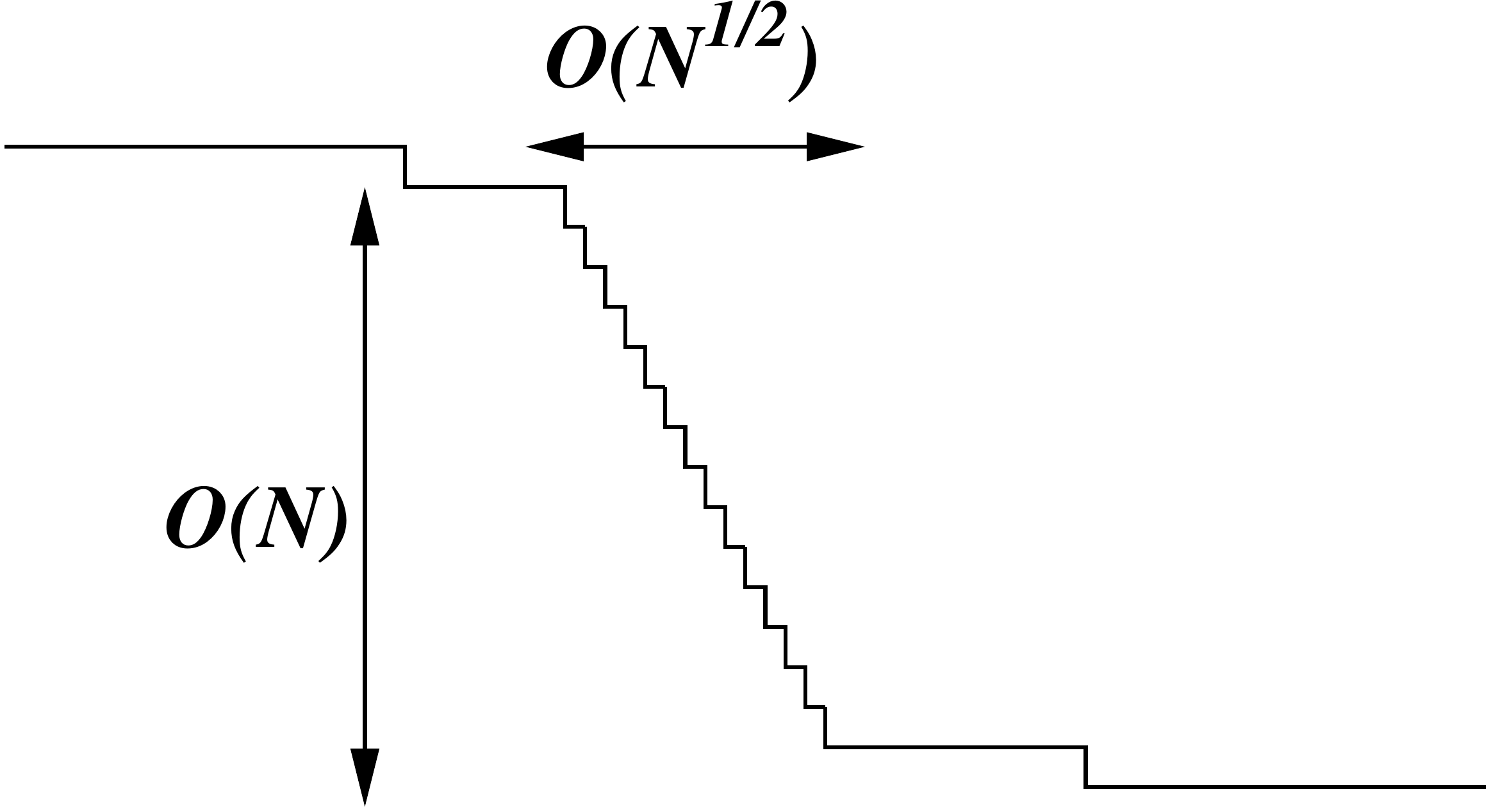}}
  \end{center}
  \caption{(a) A vicinal surface with steps. The step locations and the lengths 
  between them are denoted by $x_i$ and $l_i := x_{i+1}-x_i$ 
for $i\in\mathbb{Z}$.
  (b) The ``almost linear'' shape of a step bunch for $m=0, n=2$ in \cite{Luo2016-p737-771}.}
  \label{figure..step}
\end{figure}
$F_{\mathrm{ad}}$ is adatom flux, and $B:=a^2\frac{\rho_0 D}{k_B T}$, with $a$, $\rho_0$, $D$, $k_B$ and $T$ being the lattice constant, the equilibrium adatom density on a step in the absence of elastic interactions, the diffusion constant on the terrace, Boltzmann constant and temperature, respectively. The crucial quantity is the chemical potential 
\begin{equation}\label{eq..mu_i}
  \mu_i=-\sum_{j\neq i}\left(\frac{\alpha_1}{x_j-x_i}-\frac{\alpha_2}{(x_j-x_i)^3}\right),
\end{equation}
where the $\alpha_1$-term is the force monopole which is attractive 
while the $\alpha_2$-term is the force dipole which is repulsive. 
Their physical origins are lattice misfit and broken bond effects.
Note that the monopole decays much more slowly in space. 
It is the key driving force for the step bunching phenomena.
If we define $E$ to be the following elastic energy of a step configuration,
\begin{equation}\label{eq..LJ02Energy}
  E=\sum_{j>i,j\in\mathbb{Z}}\sum_{i\in\mathbb{Z}}\left(\alpha_1\log|x_j-x_i|+\frac{\alpha_2}{2(x_j-x_i)^2}\right),
\end{equation}
then $\mu_i=\frac{\partial E}{\partial x_i}$
so that \eqref{eq..TersoffOriginal} can be interpreted as the gradient
flow of $E$ with respect to an appropriate metric 
on the step configuration space.

The results in \cite{Luo2016-p737-771} are roughly stated as follows.
For a system with reference system length scale of order $N$, we have 
obtained scaling laws for 
(1) the minimum energy: $ \min E\sim  N^2\log N$;
(2) minimal terrace length: $\min_i (x_{i+1}-x_i)\sim N^{-1/2}$; and
(3) bunch width (system size): $x_N-x_1\sim N^{1/2}$.
The asymptotics are valid in the limit $N\rightarrow\infty$. 
They demonstrate the appearance of bunching phenomena and describe 
quantitatively the shape of the step bunches
(cf. Fig. \ref{figure..step}(b)). 

A natural question to ask is: what is special about the pairwise interactions 
in the epitaxial growth model \eqref{eq..LJ02Energy}? 
It seems that the interaction between steps in Eq. \eqref{eq..mu_i} is similar 
to the force between the classical Lennard--Jones (LJ) $(6,12)$ interaction. 
Will the step-bunching-like phenomenon appear in a particle system governed by the well-known LJ $(6,12)$ potential? If not, what is the difference between the two cases? Another interesting question is whether the step-bunching phenomenon depends on the interaction range. These are relevant questions since in numerical simulations we often truncate and regularize the LJ potential or other classical multi-particle interactions. The goal of this paper is to consider in general the competition between the attraction 
and repulsion effects and the interaction range and investigate
how they determine the final pattern formation.

To be more precise, in the present work, we study a generalized step model 
in one dimension with pair potential given by
\begin{equation}
  V(x)=\left\{
  \begin{array}{ll}
      -\frac{\alpha_1}{m}|x|^{-m}+\frac{\alpha_2}{n}|x|^{-n}, & -1<m<n,\, m\neq 0,\, n\neq 0, \\
      \alpha_1\log|x|+\frac{\alpha_2}{n}|x|^{-n}, & 0=m<n, \\
      -\frac{\alpha_1}{m}|x|^{-m}-\alpha_2\log|x|, & -1<m<n=0,
  \end{array}
  \right.\label{eq..PairPotential.mn}
\end{equation}
where $m$ and $n$ are exponents for the interaction strength. Indeed, this is 
the simplest but still informative model which incorporates both attractive and repulsive interactions. In essence, $m$ and $n$ characterize both the 
singularity and the decaying rate of the pair interaction potential
between the steps. 
The condition $m<n$ is to guarantee that the function
$V(\cdot)$ is single-welled, i.e., it has only one global minimum
while the restriction $-1 < m$  is to make sure that the force
$V'(\cdot)$ goes to zero as $x \rightarrow \infty$. 
More precisely, for $x>0$ and $-1 < m < n$, we have
\begin{eqnarray*}
V'(x) = \frac{\alpha_1}{x^{m+1}} - \frac{\alpha_2}{x^{n+1}}
\,\,\,\,\,\text{and}\,\,\,\,\,\,
V''(x)= -\frac{\alpha_1(m+1)}{x^{m+2}} + \frac{\alpha_2(n+1)}{x^{n+2}}
\end{eqnarray*}
so that there is only one critical point $l_*$ and
inflection point $l_{**}$ of $V$:
\[
l_* = \left(\frac{\alpha_2}{\alpha_1}\right)^{\frac{1}{n-m}}
\,\,\,\,\,\text{and}\,\,\,\,\,\,
l_{**} = \left(\frac{\alpha_2(n+1)}{\alpha_1(m+1)}\right)^{\frac{1}{n-m}}.
\]
(See also Fig. \ref{figure..pair.potential} for some illustration.)

For a finite system with $N$ steps, we also want to consider the effective 
range of interaction between the steps. 
Such a consideration can be modeled by an energy functional of the following form:
\begin{equation}\label{eq..Energy.mnalphaDefinition}
  E[Y_N]=\sum\limits_{1\leq i<j\leq N,\,\,j-i\leq \lfloor N^{\alpha}\rfloor}V(y_j-y_i),
\end{equation}
where $Y_N=(y_1,\cdots,y_N)^T$ with $y_i < y_j$ for $i < j$.
Note that steps $i$ and $j$ interact if $|i-j|\leq N^\alpha$.
We call the above {\em generalized LJ $(m,n)$ model}.

\begin{figure}
  \begin{center}
  (a)\subfigure{\includegraphics[width=0.45\textwidth]{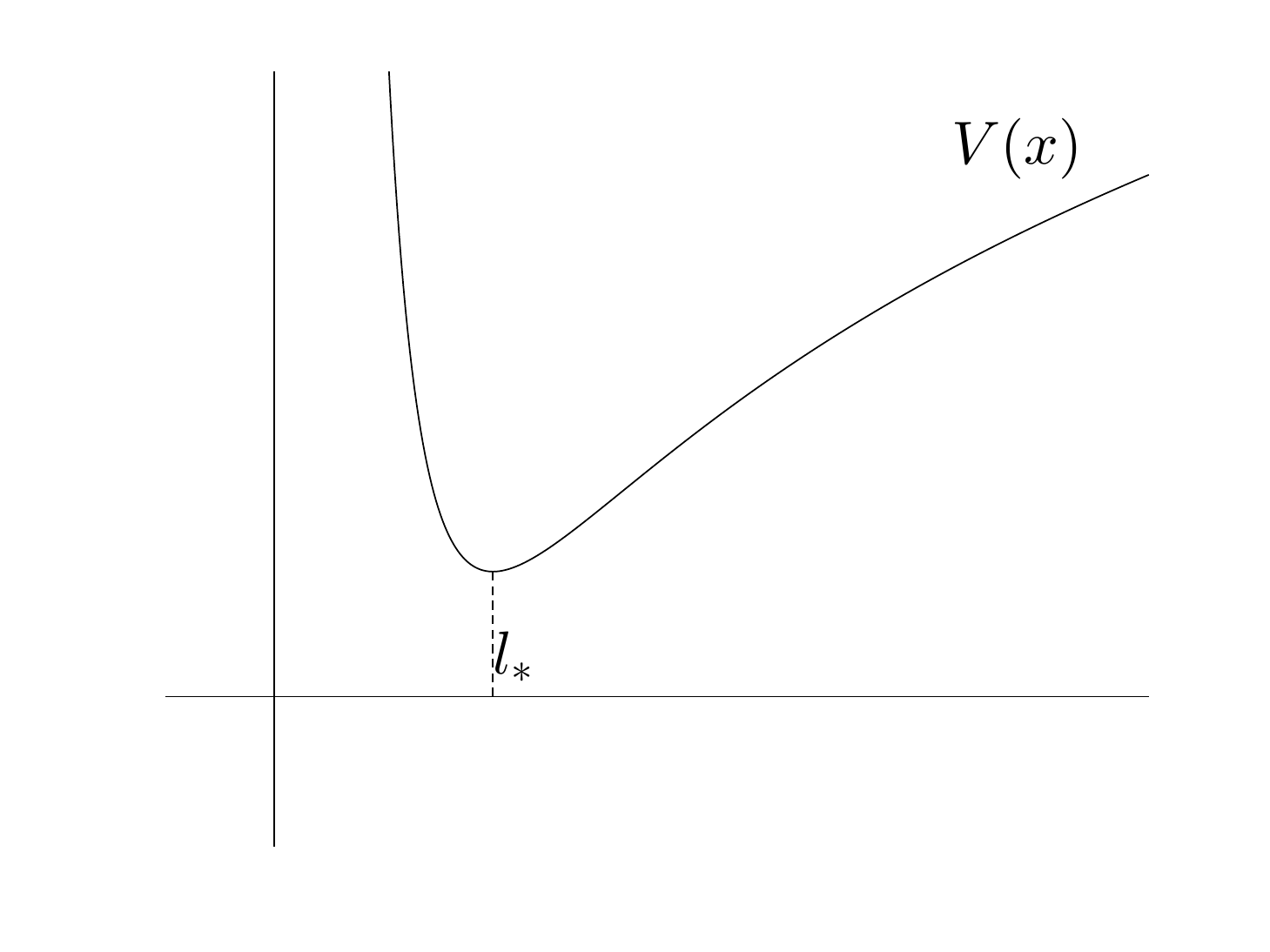}}
  (b)\subfigure{\includegraphics[width=0.45\textwidth]{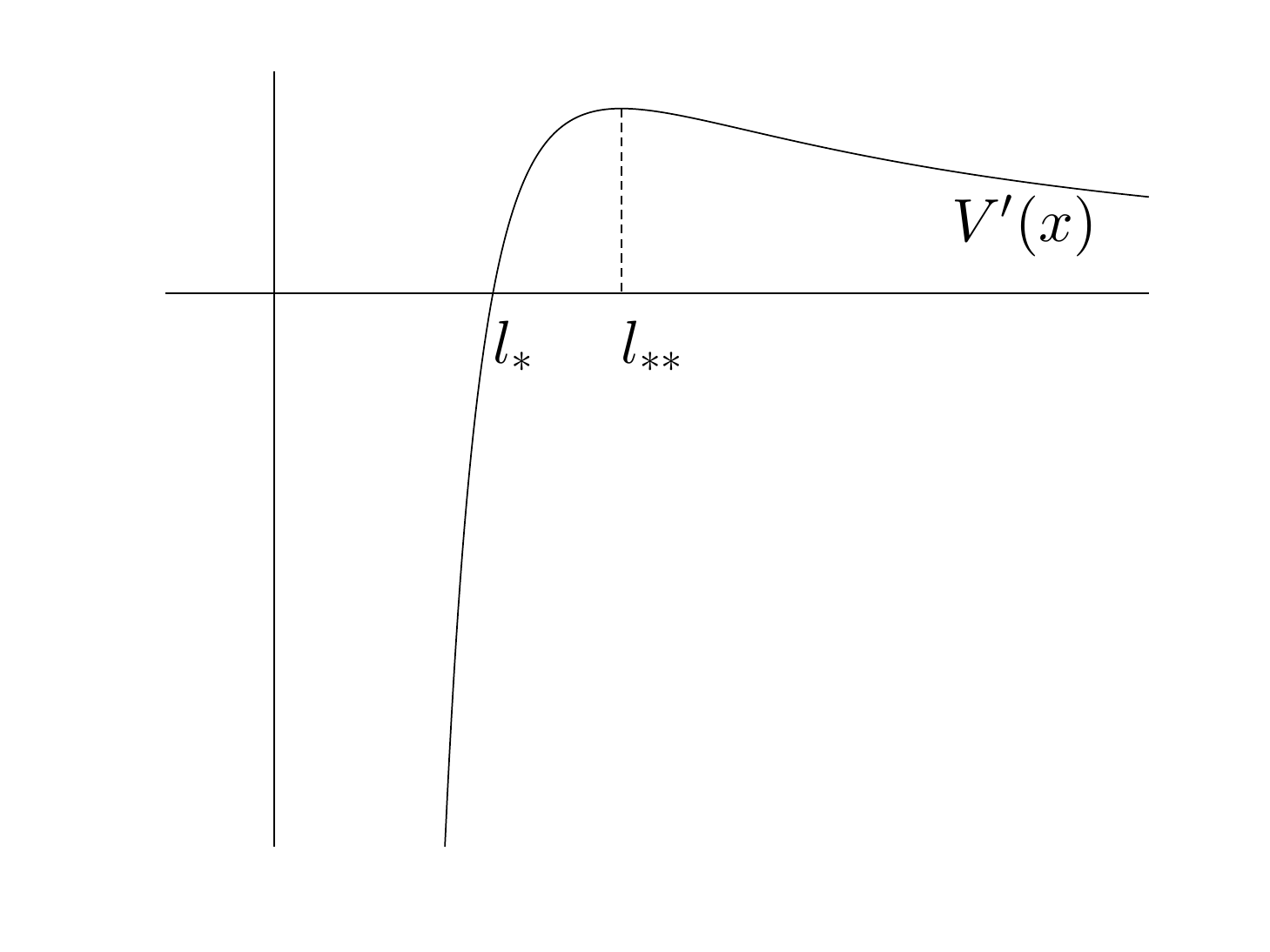}}
  (c)\subfigure{\includegraphics[width=0.45\textwidth]{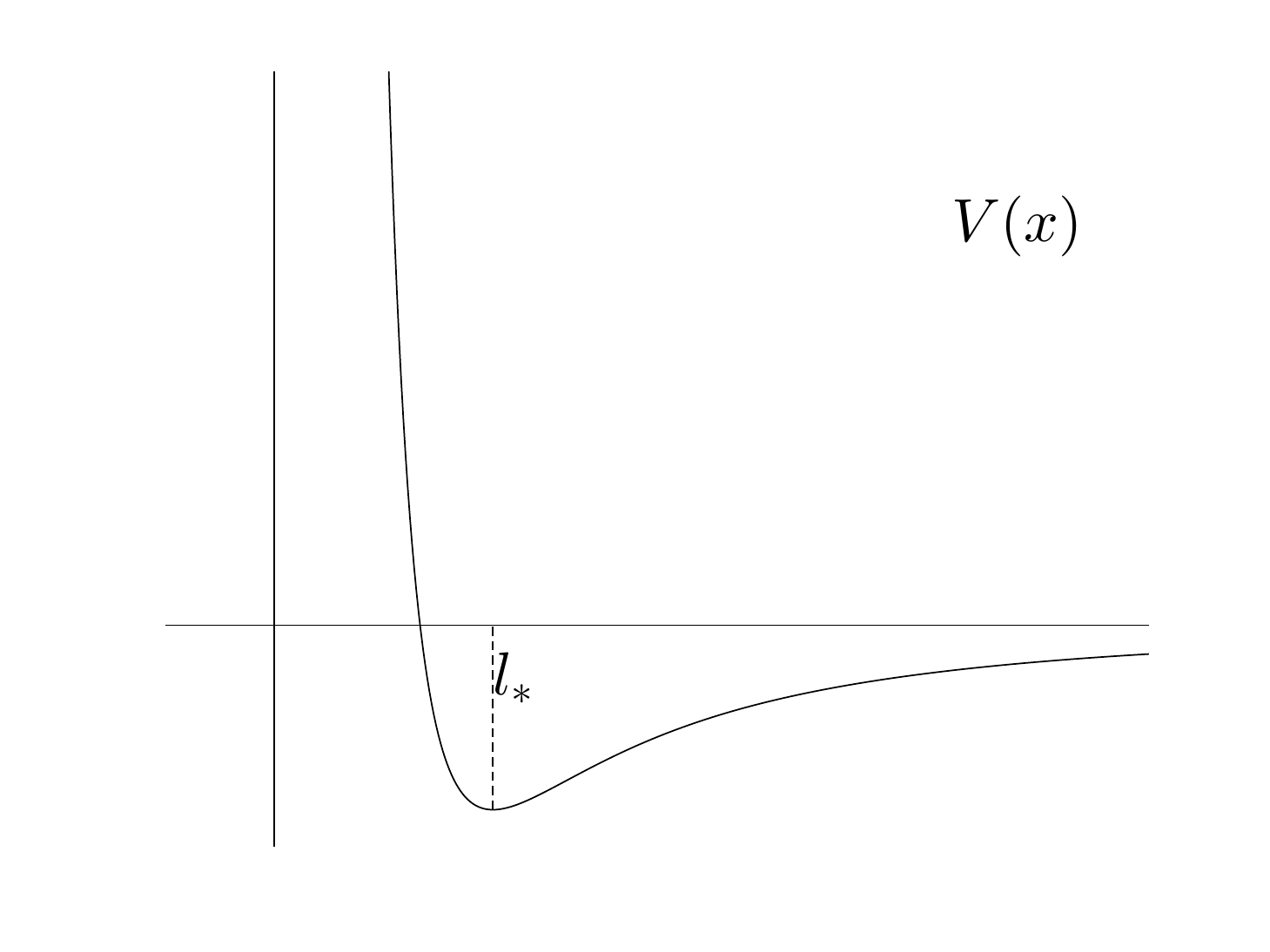}}
  (d)\subfigure{\includegraphics[width=0.45\textwidth]{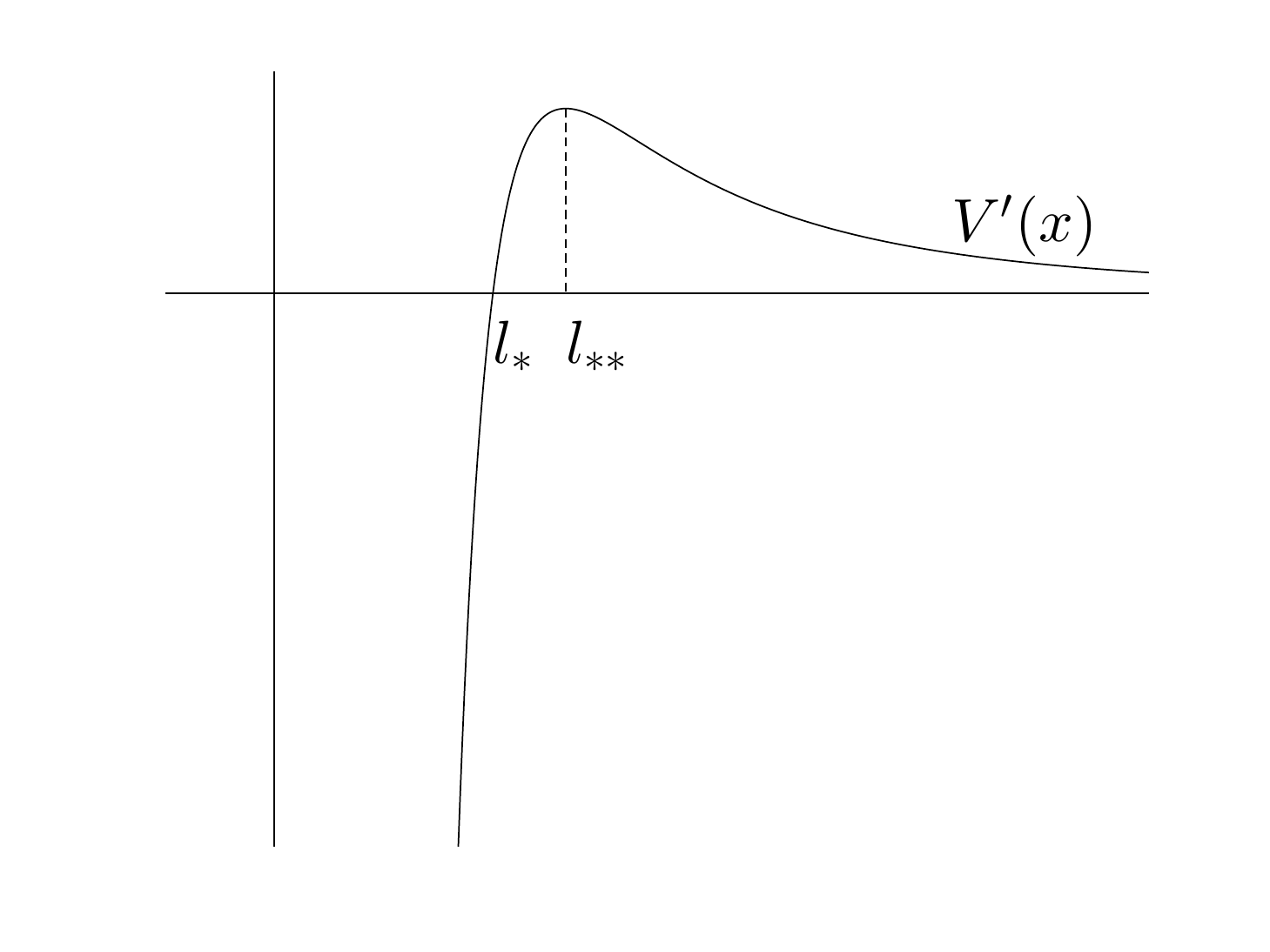}}
  \end{center}
  \caption{Pair potential $V(x)$ and its derivative $V'(x)$ in different 
regimes. (a) and (b): $-1<m<0$ and $1<n$; 
(c) and (d): $0<m<n$. 
Note that for $m< 0$, $V(x)$ grows to infinity as $x\rightarrow\infty$
while for $0< m$, $V(x)$ decays to zero as $x\rightarrow\infty$.}
\label{figure..pair.potential}
\end{figure}

The present paper considers the minimization problem of the energy
\eqref{eq..Energy.mnalphaDefinition},
i.e., to find an $X_N=(x_1,\cdots,x_N)^T$ with $x_1<x_2<\cdots<x_N$ such that
\begin{equation} 
E[X_N]=\inf_{y_1<y_2<\cdots<y_N}E[Y_N].\label{eq..EnergyMinimizationProblem}
\end{equation}
We still call the functional \eqref{eq..Energy.mnalphaDefinition}
an ``epitaxial growth model''
so as to be consistent with the terminology of the previous work 
\cite{Luo2016-p737-771} where the interaction is the LJ $(0,2)$ potential. Note that if $\alpha=0$, then the interaction is 
{\em nearest-neighbor} which will be shown not to have 
any bunching instability. 
On the other hand, if $\alpha>0$, then the interaction is nonlocal. In particular, $\alpha=1$ corresponds to fully nonlocal interaction. For the nonlocal case ($\alpha>0$), we will show that bunching instability takes place if $-1<m<1<n$.

From physical experience, in an LJ $(6,12)$ system, the system size $x_N-x_1$ grows linearly in $N$. In fact, this is related to the crystallization problem which asks whether under appropriate conditions, 
the perfect {\em periodic} lattice 
configuration is the minimizer as $N$ tends to infinity.
In one dimension, especially for LJ type potential, the crystallization problem is completely understood \cite{Blanc2015-p-}. In \cite{Ventevogel1978-p343-361}, Ventevogel proved that the lattice structure gives the minimum energy 
for the LJ $(m,n)$ potential (Eq. \eqref{eq..Energy.mnalphaDefinition}) with $1<m<n$. 

As far as we know, all the results on the crystallization problem exclude 
the case $m<1$ since it is not physically relevant for the models considered
in those works. 
However, it seems that the regime $m<1$ is where step-bunching takes place. 
One of our main results of this paper is the demonstration of bunching 
phenomenon for $-1<m<1<n$ but non-bunching for $1<m<n$.
This is carried out by means of a unified approach. 
To the best of our knowledge, together with \cite{Luo2016-p737-771}, our works are the first to give a quantitative description of the bunching phenomena.

These questions are by no means trivial in terms of rigorous analysis or even intuitively speaking. 
The difficulties come from the long-range interaction and the discreteness of 
the model. On a technical level, we need to deal with the double summation
in the formulation \eqref{eq..Energy.mnalphaDefinition}. 
We will improve the idea and techniques used in \cite{Luo2016-p737-771} to 
estimate the force exerted by a step chain which is essentially a first
order information. 
This leads to a sharp lower bound for the terrace length. 
All the other estimates are based on this lower bound. Interestingly, some
of our results are applicable even for general critical points.

\section{Main Results}

As mentioned before, this work focuses on the investigation of the step 
bunching phenomenon in the fully nonlinear setting. 
In this case, it is more convenient to analyze the step 
bunching phenomena with Neumann boundary condition which is the natural 
boundary condition for finitely many ($N$) steps.
(This can be extended to the periodic setting by the concentration-in-half-period technique as in \cite[Theorem 3(b)]{Luo2016-p737-771}.)

%\emph{Step bunching phenomenon---nonlinear regime}
We consider here the nonlinear energy minimization problem 
\eqref{eq..EnergyMinimizationProblem} and 
study the properties of its minimizers. 
Without loss of generality, we assume $\alpha_1=\alpha_2=1$ in the rest of 
the paper. We remark that $m, n$ can be non-integers.
The existence of a minimizer follows from the continuity of $V$ and 
its behavior as $x$ approaching $0$ or $\infty$. The proof is omitted here 
as it is very similar to \cite[Theorem 1(b)]{Luo2016-p737-771}. 
In this work, we will estimate the {\em minimum energy}
\begin{equation}\label{minE}
  E_N:=\min_{y_1<y_2<\cdots<y_N}E[Y_N]
\end{equation}
and investigate the asymptotic behavior of minimizers or even critical points 
$X_N$ and their dependence on $m$, $n$, and $\alpha$ as $N\rightarrow+\infty$. 

Here we make a remark about the notation. 
For simplicity, $X_N$ can refer to a general
step configuration, a minimizer, or a critical point
and its energy is denoted by $E[X_N]$.
The meaning of $X_N$ will be clear or specified in the context it appears. 
However, $E_N$ will always refer to the minimum energy as 
defined in \eqref{minE}.

In order to give precise statements, we first introduce the following
quantities.
\begin{defi}[minimal terrace length $\lambda_N$ and system size $w_N$]
For any step configuration $X_N=(x_1,\cdots,x_N)^T$ with $ x_1<\cdots< x_{N}$,
we define
  \begin{align}
    \lambda_N
    &:= \min_{1\leq i\leq N-1}\{x_{i+1}-x_i\},\\
    w_N
    &:= x_N-x_1.
  \end{align}
\end{defi}

Next we define the bunching phenomenon as follows. 
\begin{defi}[Bunching/Non-bunching regime]
  We say a sequence of configurations $\{X_N\}_{N=1}^{\infty}$ is a bunching (respectively non-bunching) sequence, if
  \begin{equation}
      \limsup_{N\rightarrow+\infty}\frac{w_N}{N}=0\,\,\,
      \left(\text{respectively}\,\,\liminf_{N\rightarrow+\infty}\frac{w_N}{N}>0\right).
  \end{equation}
  We say that the system with parameters $m,n$, and $\alpha$ is in the bunching (respectively non-bunching) regime if any sequence of energy minimizers $\{X_N\}_{N=1}^{\infty}$ is a bunching (respectively non-bunching) sequence.
\end{defi}

For convenience, we state the following condition which will be used
very often in the paper.
A step configuration $X_N=(x_1,\cdots,x_N)^T$ with $ x_1<\cdots< x_{N}$
is said to satisfy the {\em force balance} if 
$\displaystyle 
\frac{\partial E}{\partial x_i} = 0$, for all $i=1,\cdots,N$.
Expressed in terms of the length variables, this condition becomes
\begin{equation}\label{ForceBalanceCond}
    \sum\limits_{j\leq i\leq k, k-j\leq \lfloor N^{\alpha}\rfloor-1}V'(l_j+\cdots+l_i+\cdots+l_k)=0.
\end{equation}
\noindent
The above is clearly the same as the vanishing of the first variation of 
the energy $E$ at any minimizer or critical point.

\begin{rmk}\hspace{5pt}
\begin{enumerate}
\item
Note that as critical points or even minimizers 
might not be unique, the quantities $\lambda_N$ and $w_N$ in general will 
depend on the particular step configuration $X_N$. 
However, the main point of our
results is that their dependence on $N$ are all asymptotically the same.

\item
Our definition of bunching v.s. non-bunching is somewhat different from the 
terminology used in the literature where people often consider kinetic effects 
in the bunching phenomenon. They sometimes say that bunching occurs if
a large number of steps concentrate in a region which is much narrower
compared with the initial configuration \cite{Krasteva2016-p220015-220015}. 
On the other hand, we focus on whether the average terrace length is asymptotically zero or not as $N\rightarrow\infty$. 
A non-bunching sequence, in our sense $w_N/N\rightarrow l_\infty>0$, may be regarded as a bunching sequence in the literature if 
$l_\infty\ll l_0$ where $l_0$ is the average initial terrace length.
\end{enumerate}
\end{rmk}

By means of an appropriate ansatz, 
we can obtain a heuristic 
scaling law for the minimum energy (which is in fact a rigorous 
{\em upper bound} for the energy) 
and also the underlying shape of the
step bunch. This is performed in Section \ref{sec..heuristic.scaling.laws}.
The results, illustrated in the phase diagram Fig. \ref{figure..PhaseDiagram},
can already reveal two interesting physical regimes 
which we will concentrate on in this paper.
One is $-1<m<1<n$ corresponding to the potential appeared in the epitaxial growth model \cite{Tersoff1995-p2730-2733}. The other is $1<m<n$ corresponding to the classical LJ type potential. The bunching phenomenon appears in the former but not the latter.
The main purpose of this paper is to give a quantitative description of 
the two regimes and the transition between them.

\begin{thm}[bunching regime]\label{thm..Bunching}
  Let $-1<m<1<n$ and $0<\alpha\leq 1$.
  Then there exist positive constants $C$, $C'$, $\beta$, and $N_0$ such that, for any $N>N_0$ and minimizer $X_N$ of the energy functional 
  \eqref{eq..Energy.mnalphaDefinition}, the following hold.

  (A) energy scaling law
{\small
  \begin{align*}
    CN^{1+\frac{(1-m)n\alpha}{n-m}}
    &\leq 
    E_N \leq C' N^{1+\frac{(1-m)n\alpha}{n-m}},\quad -1<m<0,\\
    \textstyle \frac{(n-1)\alpha}{n} N^{1+\alpha}\log N-C N^{1+\alpha}\log\log N
    &\leq 
    E_N \leq\frac{(n-1)\alpha}{n}N^{1+\alpha}\log N+C'N^{1+\alpha},\quad m=0,\alpha<1;\\
    \textstyle \frac{n-1}{2n} N^2\log N-C N^2\log\log N
    &\leq 
    E_N \leq \frac{n-1}{2n}N^2\log N+C'N^2,\quad m=0,\alpha=1,\\
-CN^{1+\frac{(1-m)n\alpha}{n-m}}
    &\leq 
E_N \leq -C' N^{1+\frac{(1-m)n\alpha}{n-m}},\quad 0<m<1.
  \end{align*}
}

  (B) minimal terrace length
  \begin{align}
    CN^{-\frac{(1-m)\alpha}{n-m}}
    &\leq \lambda_N
    \leq C' N^{-\frac{(1-m)\alpha}{n-m}},\quad m\neq 0,\\
    CN^{-\frac{\alpha}{n}}\left(\log N\right)^{-\frac{1}{n}}
    &\leq \lambda_N
    \leq C' N^{-\frac{\alpha}{n}}, \quad m=0;
  \end{align}

  (C) system size
  \begin{align}
    CN^{1-\frac{(1-m)\alpha}{n-m}}
    &\leq w_N \leq C' N^{1-\beta},\quad m\neq 0,\\
    CN^{1-\frac{\alpha}{n}}\left(\log N\right)^{-\frac{1}{n}}
    &\leq w_N \leq C' N^{1-\beta},\quad m=0.
  \end{align}
  In particular, we have $\lambda_N\ll 1$ and the system is in the 
  bunching regime.
\end{thm}

\begin{thm}[non-bunching regime]\label{thm..Non-bunching}
  Suppose that either 
(i) $1<m<n$, $0 < \alpha \leq 1$ or 
(ii) $-1<m<n$, $\alpha=0$.
  There exist positive constants $C$, $C'$, and $N_0$ such that, for any $N>N_0$ and minimizer $X_N$ of the energy functional 
  \eqref{eq..Energy.mnalphaDefinition}, the following hold.

  (A) energy scaling law
  \begin{equation}
    \left\{
\begin{array}{lc}
\text{case (i):} & -CN\leq E_N\leq -C'N; \\
\text{case (ii):} & |E_N|\leq CN; 
\end{array}
\right.
  \end{equation}

  (B) minimal terrace length
  \begin{equation}
    C\leq \lambda_N \leq C';
  \end{equation}

  (C) system size
  \begin{equation}
    CN\leq w_N \leq C'N.
  \end{equation}
  In particular, we have $\lambda_N=O(1)$ and the system is in the non-bunching regime.
\end{thm}
\noindent
We remark that that the various constants $C$ and $C'$ in
the above may be 
different in different parts of the statements. In general, they can 
depend on $n, m$, and $\alpha$, but not on the system size $N$.

%Some remarks are in place.

\begin{rmk} \hspace{5pt}
\begin{enumerate}
\item We note again that the epitaxial growth model ($m=0$ and $n=2$) 
belongs to the bunching regime (Theorem \ref{thm..Bunching}) while
the classical Lennard--Jones model ($m=6$ and $n=12$) belongs to the
non-bunching regime (Theorem \ref{thm..Non-bunching}).

\item
All of the statements in Theorem \ref{thm..Bunching}
involve an exponent with value $\frac{1-m}{n-m}$
which is a decreasing function of $m$ and $n$ in the region $-1 < m < 1 < n$. 
This leads to that the scaling for the minimum terrace length $\lambda_N$
is an increasing function of $m$. In a sense, the bunching phenomena is 
``weakened''. This is also revealed in the
dependence of $\lambda_N$ on $\alpha$: a bigger value of $\alpha$, i.e.
more nonlocal interaction, causes
a more prominent bunching effect. The case $m=0$ is critical in all 
of the above quantitative results.

\item
  The parameter ranges for $m$ and $n$ in Theorem \ref{thm..Bunching} 
are included in case (ii) of Theorem \ref{thm..Non-bunching}. 
The difference appears in their values of $\alpha$.
In particular, $\alpha=0$ 
corresponds to finite-range interaction (and in fact nearest neighbor 
as $N^\alpha = 1$)
which does not lead to any bunching phenomenon.

\item
The upper and lower bounds in part (C) of Theorem \ref{thm..Bunching} do
not match in general. But we can still provide some partial results 
where the bounds do match (see Proposition \ref{thm..size.system.bunching.alpha=1.further.result} for the case of $-1<m\leq 0,\,\,1<n$, and $\alpha=1$). 
In any case, the positive $\beta$ in part (C) of Theorem 
\ref{thm..Bunching} indicates that the system is in the bunching regime.

\item
Our intuition leads us to believe that the bunch shape is
roughly linear so that $w_N \sim N\lambda_N$. 
This is related to the fact that the optimal energy scaling 
is the same as that given by the uniform step train ansatz 
(see Section \ref{sec..heuristic.scaling.laws}).
Such a linear shape is rigorously proved in 
\cite[Theorems 4 and 5]{Luo2016-p737-771} for the 
$(m=0,n=2)$ elasticity model.
Our current work extends this description to our generalized LJ $(m,n)$ model
to the regime $-1 < m < 0$, $1< n$ 
and $\alpha = 1$ (see Proposition 
\ref{thm..size.system.bunching.alpha=1.further.result}).
In the non-bunching case, such a statement is very much 
related to the well-known crystallization conjecture --- 
ground states are believed to be periodic in infinite extent.
In the one dimensional case, this is solved in several works 
\cite{Ventevogel1979-p274-288,Ventevogel1979-p569-580,Ventevogel1978-p343-361}.

\item 
For technical reasons, for the parameter regime 
$-1<m<n<1$, 
we are only able to provide an upper bound for the minimum energy $E_N$.
Heuristically, the result says that the system is in the bunching
regime, and in fact, of the ``most severe'' type.
See the statement and discussion in Section 
\ref{sec..heuristic.scaling.laws}, in particular,
the region A in Fig. \ref{figure..PhaseDiagram}.

\item
Note that each of the exponents $m$ and $n$ determines simultaneously 
the behaviors of the potential $V$ for $x\ll1$ and $x\gg 1$.
We believe that Theorems \ref{thm..Bunching} and \ref{thm..Non-bunching} 
also work for more general potential $V(x)$, not necessarily of Lennard--Jones type \eqref{eq..PairPotential.mn} but still
satisfying $V(x)\sim |x|^{-m}$ asymptotically for $|x|\gg 1$ 
and $V(x)\sim |x|^{-n}$ for $|x|\ll 1$.
\end{enumerate}
\end{rmk}

Here we compare our results and technique of proof with some works on the 
crystallization problem. 
Refs. \cite{Ventevogel1978-p343-361} and \cite{Gardner1979-p719-724} show the 
crystallization phenomenon for a one-dimensional system with Lennard--Jones $(m,n)$ 
interaction for $1<m<n$. 
Refs. \cite{Ventevogel1979-p274-288} and \cite{Ventevogel1979-p569-580} 
are further extensions of \cite{Ventevogel1978-p343-361}. 
All of them use {\em energetic consideration} of the energy functional $E$ 
\eqref{eq..Energy.mnalphaDefinition}.
The calculation involves 
careful rearrangement argument for the 
double summation of the pair potential $V$. 
However, some of our statements
(for example, the lower bound of minimal terrace length, 
and upper bound for the bunching size) hold for any critical point of $E$, 
not just for minimizers. To achieve these, we make use of the 
force balance condition \eqref{ForceBalanceCond}
which only involves a single summation.
These are summarized and formulated in the following 
corollaries. The
results are new and different from all the previous works.
\begin{cor}\label{thm..BunchingCriticalPoint}
  Suppose that $-1<m<1<n$, and $0<\alpha\leq 1$. Then there exist positive constants $C$, $\beta$, and $N_0$ such that for all $N>N_0$ and all 
  {\em critical point} $X_N=(x_1,\cdots,x_N)^T$ of \eqref{eq..Energy.mnalphaDefinition}, we have
  \begin{equation*}
    w_N\leq C N^{1-\beta}.
  \end{equation*}
  In particular, {\em any sequence of critical points $\{X_N\}_{N=1}^\infty$
  is a bunching sequence.}
  {\em (The estimate here are exactly the same as
  the upper bound for $w_N$ in Theorem \ref{thm..Bunching}(C).)}
\end{cor}
\begin{cor}\label{thm..NonBunchingCriticalPoint}
%Suppose that $-1 < m$ and $1 < n$ and $0 \leq \alpha \leq 1$, 
The following estimates for the minimal terrace length $\lambda_N$
and the system size $w_N$ hold for any {\em critical point} $X_N$ of the
energy $E$ \eqref{eq..Energy.mnalphaDefinition}:
  \begin{equation*}
    l_*\geq \lambda_N\geq
    \left\{
    \begin{array}{ll}
      CN^{-\frac{(1-m)\alpha}{n-m}}, & -1<m<1<n,m\neq 0, 0<\alpha\leq 1,\\
      CN^{-\frac{\alpha}{n}}(\log N)^{-\frac{1}{n}}, & m=0, 1< n, 0<\alpha\leq 1,\\
      C, & 1<m<n, 0<\alpha\leq 1,\\
      C, & -1<m<n, \alpha=0;
    \end{array}
    \right.
  \end{equation*}
  \begin{equation*}
    (N-1)l_*\geq w_N\geq
    \left\{
    \begin{array}{ll}
      CN^{1-\frac{(1-m)\alpha}{n-m}}, & -1<m<1<n, m\neq 0, 0<\alpha\leq 1,\\
      CN^{1-\frac{\alpha}{n}}(\log N)^{-\frac{1}{n}}, & m=0, 1< n, 0<\alpha\leq 1,\\
      CN, & 1<m<n, 0<\alpha\leq 1,\\
      CN, & -1<m<n, \alpha=0.
    \end{array}
    \right.
  \end{equation*}
In particular, if ($1<m<n$, $0<\alpha\leq 1$) or ($-1<m<n$, $1<n$, $\alpha=0$), then {\em any sequence of critical points $\{X_N\}_{N=1}^\infty$ 
is a non-bunching sequence}.
  {\em (The estimates here are exactly the same as the
  lower bounds for $\lambda_N$ and $w_N$ in 
  Theorems \ref{thm..Bunching}(B, C) and \ref{thm..Non-bunching}(B, C).)}
\end{cor}

The rest of the paper is outlined as follows.  
In Section \ref{sec..heuristic.scaling.laws}, 
we give an upper bound for the minimum energy $E_N$ in various regimes
(Theorem \ref{thm..EnergyUpperBound}).
The results are illustrated in the phase diagram 
(Fig. \ref{figure..PhaseDiagram}) which heuristically 
reveal the bunching and non-bunching regimes and the transition between them.
Theorems \ref{thm..Bunching} and \ref{thm..Non-bunching} are proved in 
Sections 
\ref{sec..lower.bounds.terrace.lengths}--\ref{sec..Proof-Non-bunching}.
Given the proofs, Corollaries \ref{thm..BunchingCriticalPoint} 
and \ref{thm..NonBunchingCriticalPoint} follow immediately.
In the Appendix, Section \ref{ProofEnergyUpperBound}, 
we prove Theorem \ref{thm..EnergyUpperBound} 
which involves quite detail and elaborate calculations.

\section{Upper Bounds for the Minimum Energy $E_N$ and 
Phase Diagram}\label{sec..heuristic.scaling.laws}
Inspired by the numerical evidence in \cite{Tersoff1995-p2730-2733} and 
the analytical results from \cite{Luo2016-p737-771}, we anticipate that 
the optimal height profile in our general epitaxial growth model \eqref{eq..Energy.mnalphaDefinition} is almost a uniform step train consisting of a 
series of consecutive terraces with roughly equal lengths $l_0$. 
In this section, we would exploit such a step profile as an ansatz. 
More precisely, let $l_0$ be a positive number.
Then the uniform step train with length $l_0$ is defined as:
\begin{equation}
  X_N^{0}=(x_1^0,\cdots,x_N^0)^T,\,\,x_i^0=\left(i-1\right)l_0,\,\,i=1,2,\cdots,N.\label{eq..ansatz}
\end{equation}
By optimizing the value of $l_0$, we arrive at 
an upper bound for the minimum energy $E_N$ and 
also a first illustration about its dependence on $m$, $n$ and $\alpha$.

The result is listed in five cases corresponding to
different parameter values.

\begin{thm}[Upper bound for $E_N$]
\label{thm..EnergyUpperBound} 
For any $0\leq\alpha\leq 1$ and $-1<m<n$, 
(there exists an $l_0$ such that)
the following upper bounds for $E[X_N^0]$ hold.
(Again, in the following the constants $C$ can depend on 
$m, n$ and $\alpha$ but not on $N$.)
\begin{description}
\item[Case A: $-1 < m < n < 1$, $0 < \alpha$.]
\begin{align*}
  E[X_N^0]
  \leq 
  \left\{\begin{array}{lll}
    -C N^{1+\alpha}, & mn>0,\\
    C N^{1+\alpha}, &  mn\leq 0.
  \end{array}
  \right.
\end{align*}

\item[Case B: $-1 < m < n = 1$, $0 < \alpha$.]
\begin{align*}
  E[X_N^0]
  &\leq 
  CN^{1+\alpha}\log N.
\end{align*}

\item[Case C: $-1 < m < 1 < n$, $0 < \alpha$.]
\begin{align*}
  E[X_N^0]
  \leq 
  \left\{\begin{array}{ll}
    C N^{1+\frac{n(1-m)\alpha}{n-m}}, & -1<m<0,\\
    \frac{(n-1)\alpha}{n} N^{1+\alpha}\log N+CN^{1+\alpha}, & m=0, 0<\alpha<1,\\
    \frac{n-1}{2n} N^{1+\alpha}\log N+CN^2, & m=0, \alpha=1,\\
    -C N^{1+\frac{n(1-m)\alpha}{n-m}}, & 0<m<1.
  \end{array}
  \right.
\end{align*}

\item[Case D: $1 = m < 1 < n$, $0 < \alpha$.]
\begin{align*}
  E[X_N^0]\leq \textstyle -C N\log N.
\end{align*}

\item[Case E:] 
\begin{align}
  E[X_N^0] \leq
  \left\{\begin{array}{lcc}
  \text{(i):} & -C N, & 1<m<n,\,\, 0<\alpha,\\
  \text{(ii):} & CN, & -1<m<n,\,\,\alpha=0.
  \end{array}
  \right.\label{eq..heuristic.case.E}
\end{align}
\end{description}
(We remark that Cases A to D with $\alpha=0$ 
are in fact included in Case E.)
\end{thm}

In the above, it might be illustrative to consider
the chosen values of $l_0$ even though it is not needed in the statement.
Overall, we have 
$l_0 \sim N^{-\alpha}$ for Cases (A) and (B), 
$l_0 \sim N^{-\frac{\alpha(1-m)}{n-m}}$ for Case (C),
and $l_0 \sim 1$ for Cases (D) and (E).

\begin{figure}
  \subfigure{\includegraphics[width=0.8\textwidth]{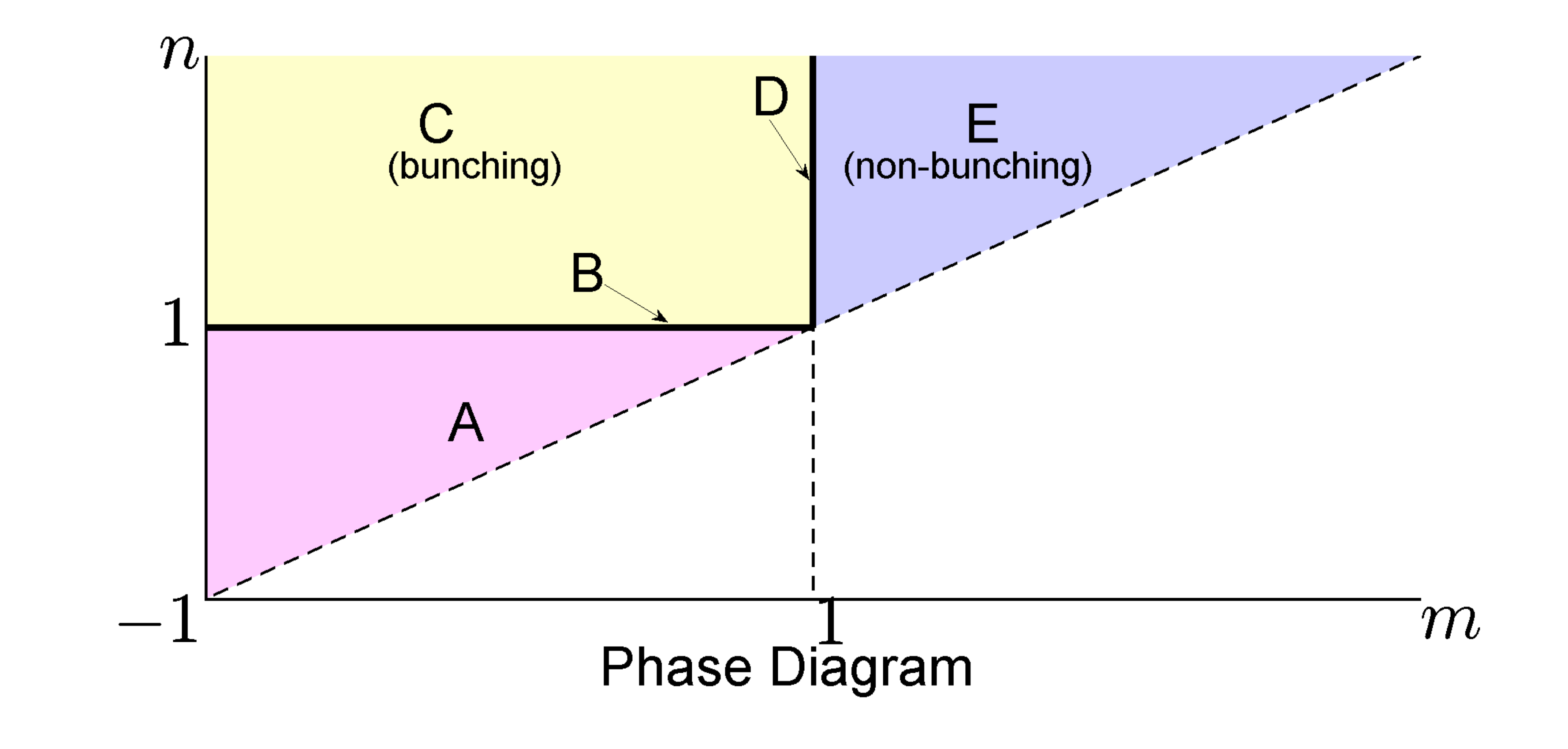}}
  \caption{Phase Diagram of the Scaling Law. This diagram characterizes the scaling behaviors of Lennard--Jones $(m,n)$ potential.}\label{figure..PhaseDiagram}
\end{figure}

The above quantitative description is illustrated in the form of a 
phase diagram for $\alpha>0$ (Figure \ref{figure..PhaseDiagram}). 
Note that bunching phenomena occurs in cases A, B, C, and D with the typical
length scale $l_0$ for the minimal terrace length gradually {\em increases}.
In other words, the bunching effect is {\em weakened}. Case A ($-1 < m < n < 1$)
corresponds to weak repulsion but strong attraction leading to the most
severe type of bunching with $l_0 \sim O(N^{-\alpha})$ (for $0<\alpha \leq 1$).
Case E ($1 < m < n$) corresponds to strong repulsion but weak attraction
leading to no bunching with $l_0 \sim O(1)$. 
We point out again that 
the Tersoff's epitaxial growth model 
\cite{Tersoff1995-p2730-2733,Liu1998-p1268-1271}
analyzed in \cite{Luo2016-p737-771} corresponds to Case C with $m=0, n=2$.
The classical Lennard-Jones potential falls in Case E with $m=6, n=12$.

The proof of the Theorem is quite involved in terms of
calculation and is thus postponed to Appendix \ref{ProofEnergyUpperBound}.
But here we remark about the scope of this paper. 
Our main contribution is a quantitative description of Cases C and E.
We will provide matching lower bounds for the minimum energy $E_N$, 
the minimum terrace length $\lambda_N$ and the system size $w_N$.
For technical reasons, our method currently cannot handle
Case A rigorously beyond the upper bound for $E_N$.
Cases B and D are critical boundary cases in the phase diagram whose
asymptotics are hard to quantify.

\section{Proof of the Theorems}
In this section, we will prove our main results:
Theorem \ref{thm..Bunching} (which covers Case C for
the bunching regime) and Theorem \ref{thm..Non-bunching} 
(which covers Case E for the non-bunching regime).
Upper bounds for the minimum energy $E_N$ for both Theorems 
are already stated in Theorem \ref{thm..EnergyUpperBound} 
in the previous section. To establish more precise information,
we will make use of 
the {\em force balance} condition \eqref{ForceBalanceCond}
(Lemma \ref{lem..step.chain.force.estimate})
and {\em minimum energy consideration}
(Lemma \ref{lem..step.chain.energy.estimate}).
It turns out that the crucial quantity is the minimum terrrace length 
$\lambda_N$. Furthermore, there is a connection between $\lambda_N$ and
a lower bound for $E_N$ (see \eqref{eq..energy.lower.bound}).
Using this relation, we are able to obtain matching lower and upper bounds for 
both $\lambda_N$ and $E_N$. 
Note that the use of force balance is applicable for any 
critical points of the energy, not just minimizers.
In particular, the proof of the upper bound for the system size $w_N$
in Section \ref{sec..UpperBoundSystemSize} fully takes advantage of 
such a consideration.

\subsection{Useful Lemmas}\label{sec..useful.lemmas}
We first obtain some a priori upper bounds for all the terrace lengths and 
the system bunch size $w_N$. The first three use force balance while the others 
use energy consideration.
\begin{prop}\label{prop..upper.bound.terrace.length}
  Suppose that $-1<m<n$ and $0\leq \alpha\leq 1$. For any $N$ and critical point $X_N=(x_1,\cdots,x_N)^T$ of \eqref{eq..Energy.mnalphaDefinition}, we have $x_{i+1}-x_i\leq l_*$ for all $i=1,2,\cdots, N-1$. In particular, 
  \begin{align}
  \lambda_N
  &:= \min_{1\leq i\leq N-1}\{x_{i+1}-x_i\}\leq l_*,\\
  w_N
  &:= x_N-x_1 \leq (N-1)l_*.
  \end{align}
\end{prop}
\begin{proof}
  We prove the statement by contradiction. Suppose for some $i$ that $l_i>l_*$.
  Then $l_j+\cdots+l_i+\cdots+l_k\geq l_i>l_*$ for all $j\leq i\leq k$. 
  Since $V'(x)>0$ for $x>l_*$, we have
  \begin{equation*}
    \sum_{j\leq i\leq k,\,\, k-j\leq \lfloor N^{\alpha}\rfloor-1}V'(l_j+\cdots+l_i+\cdots+l_k)>0
  \end{equation*}
contradicting the force balance condition \eqref{ForceBalanceCond}. 
\end{proof}

Next we show a lower bound for the terrace length. The proof is again 
based on the force balance condition.
We focus on the terrace $i$ with minimal length $l_i=\lambda_N$ and consider all interacting pairs $(j,k)$ with $j\leq i\leq k$. The result is achieved by carefully estimating all the forces related to the pair $(j,k)$. The following proposition is a first step toward the optimal lower bound of 
$\lambda_N$ and is extremely useful in the remaining part of this work.

\begin{prop}\label{prop..lower.bound.of.minimal.terrace.length.weak.version}
  Suppose that $-1<m<n$, $1<n$, and $0<\alpha\leq 1$. 
  For any $N$ and critical point $X_N=(x_1,\cdots,x_N)^T$ of \eqref{eq..Energy.mnalphaDefinition}, we have 
  $N^{-\alpha}\ll \lambda_N $ in the sense that 
  $\displaystyle \lim_{N\rightarrow\infty}\frac{N^{-\alpha}}{\lambda_N}=0$.

  In particular, there is an $N_0$ such that for $N > N_0$, it holds that
  $N^{-\alpha}\leq \lambda_N$.
\end{prop}
\begin{proof}
  Let $l_i=\lambda_N$. Again, by the force balance condition 
\eqref{ForceBalanceCond}, we have
  \begin{align*}
    0
    &= \sum_{j\leq i\leq k,\,\, k-j\leq \lfloor N^{\alpha}\rfloor-1}V'(l_j+\cdots+l_i+\cdots+l_k)\\
    &\leq V'(l_i)+\Big|\{(j,k):j\leq i\leq k, k-j\leq \lfloor N^\alpha\rfloor-1\}\Big|\max_{\xi>0}V'(\xi)\\
    &\leq \lambda_N^{-m-1}-\lambda_N^{-n-1}+CN^{2\alpha}.
  \end{align*}
  Note that for large $N$, we have $N^{-\alpha}\leq \frac{1}{2}$. 
If $\frac{1}{2} \leq \lambda_N$, then we are done. Now suppose
$\lambda_N<\frac{1}{2}$. Then we have $0\leq \Big[(\frac{1}{2})^{n-m}-1\Big]\lambda_N^{-n-1}+CN^{2\alpha}$ and thence $\lambda_N\geq CN^{-\frac{2\alpha}{n+1}}\gg N^{-\alpha}$. 
In either cases, we have the desired statement.
\end{proof}

The following important lemma gives an upper bound of the force exerted by a step chain.

\begin{lem}\label{lem..step.chain.force.estimate}
  Suppose that $-1<m<n$, $1<n$, and $0<\alpha\leq 1$. There exist $C$ and $N_0$ such that for any $N> N_0$ and critical point $X_N=(x_1,\cdots,x_N)^T$ of 
  \eqref{eq..Energy.mnalphaDefinition}, for $1\leq k\leq \lfloor N^{\alpha}\rfloor-1$, we have
  \begin{equation}
    \max\left(\sum_{i=1}^{k}V'(\xi_i)\right)
    \leq\left\{
    \begin{array}{ll}
      C \lambda_N^{-1}(N^{\alpha}\lambda_N)^{-m}, & -1<m< 0,\\
      C \lambda_N^{-1}\log N, & m=0,\\
      C \lambda_N^{-1}, & 0<m,
    \end{array}
    \right.\label{eq..step.chain.force.estimate}
  \end{equation}
  where the maximum is taken over the set 
  \begin{equation}\label{set.chain-1}
    \left\{\lambda_N\leq \xi_1<\xi_2<\cdots<\xi_k,\quad \xi_{i+1}-\xi_i\geq \lambda_N\quad\text{for}\quad i=1,2,\cdots, k-1\right\}
  \end{equation}
\end{lem}
\begin{figure}
  \begin{center}
  (a)\subfigure{\includegraphics[width=0.45\textwidth]{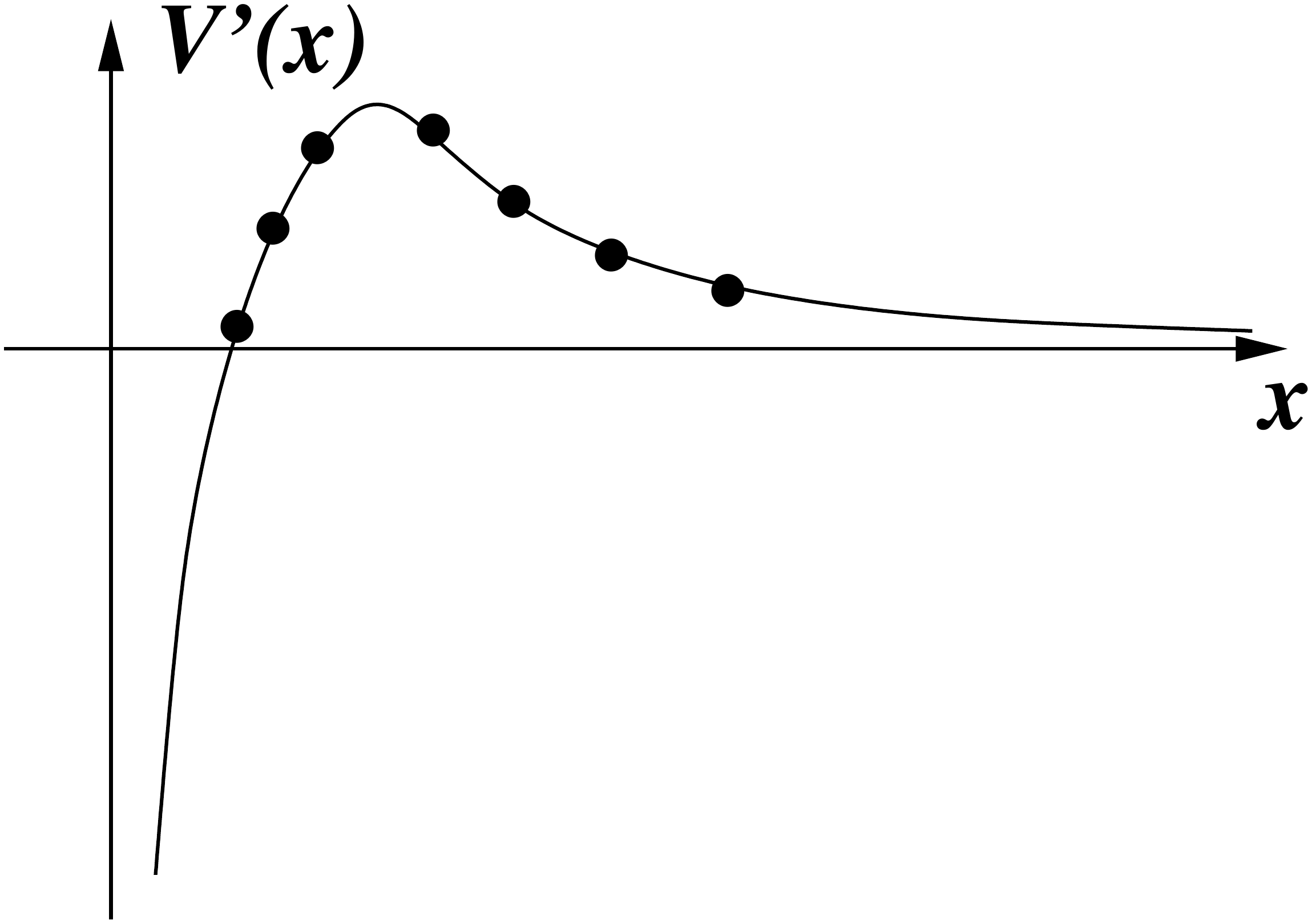}}
  (b)\subfigure{\includegraphics[width=0.45\textwidth]{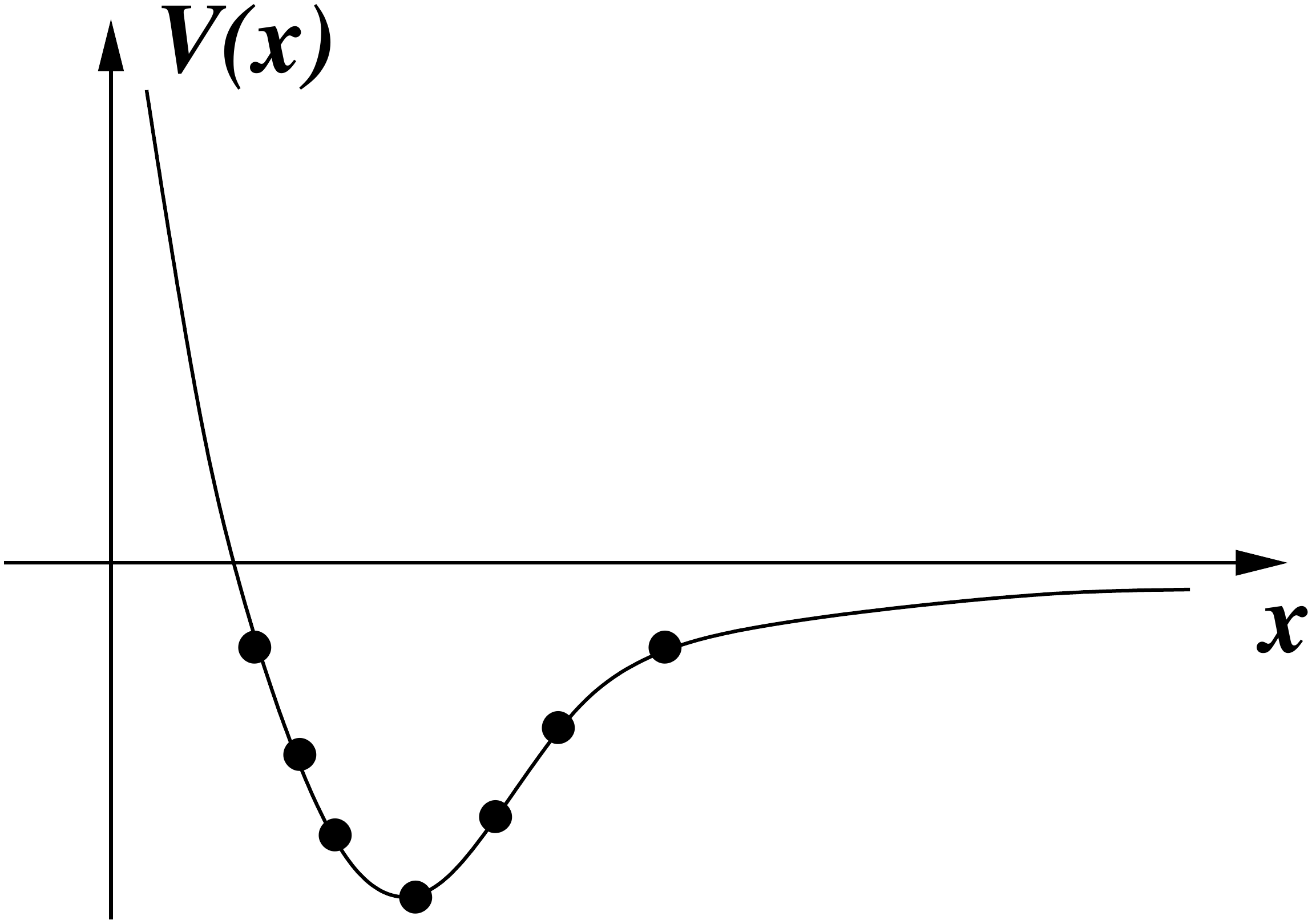}}
  \end{center}
  \caption{
(a) Maximization of chain configuration for force $V'$
(Lemma \ref{lem..step.chain.force.estimate});
(b) minimization of chain configuration for energy $V$
(Lemma \ref{lem..step.chain.energy.estimate}).}
\label{figure..chain.configuration}
\end{figure}

\begin{proof}
Let $\xi_i$'s be from the set \eqref{set.chain-1}.
  We define
  $k_1=|\{i\colon 0<\xi_i\leq l_{**}\}|$ and
  $k_2=|\{i\colon l_{**}<\xi_i\}|$.
  Then $k_1+k_2=k$ and $k_1,k_2\leq \lfloor N^{\alpha}\rfloor$. Without loss of generality, we assume that
  $k_1,k_2\geq 1$.
  (If one of them is 0, the result is still true.)
  Note that $V'(\cdot)$ is monotonically increasing on $(0,l_{**})$, and monotonically decreasing on $(l_{**},\infty)$. Then
  \begin{align*}
    \sum_{i=1}^{k} V'(\xi_i)
    &\leq \sum_{i=0}^{k_1-1}V'(l_{**}-i \lambda_N)+\sum_{i=0}^{k_2-1}V'(l_{**}+i \lambda_N)\\
    &\leq 2V'(l_{**})+\lambda_N^{-1}\int_{l_{**}-(k_1-1)\lambda_N}^{l_{**}+(k_2-1)\lambda_N}V'(x)\,\D x\\
    &= 2V'(l_{**})+\lambda_N^{-1}\left[V(l_{**}+(k_2-1)\lambda_N)-V(l_{**}-(k_1-1)\lambda_N)\right]\\
    &\leq C+\lambda_N^{-1}\left[V(l_{**}+(k_2-1)\lambda_N)-V(l_*)\right]\\
    &\leq C\lambda_N^{-1}+\lambda_N^{-1}V((l_{**}+1)N^\alpha \lambda_N),
  \end{align*}
  where in the last inequality we have used the facts that
$1\leq \lambda_N^{-1}$ and $l_{**}+(k_2-1)\lambda_N\leq (l_{**}+1)N^\alpha \lambda_N$ which follow from Propositions \ref{prop..upper.bound.terrace.length} and \ref{prop..lower.bound.of.minimal.terrace.length.weak.version},
respectively. 
Now \eqref{eq..step.chain.force.estimate} follows since for
sufficiently large $N$, we have
  \begin{equation*}
    V((l_{**}+1)N^\alpha \lambda_N)
    \leq\left\{
    \begin{array}{ll}
      C (N^{\alpha}\lambda_N)^{-m}, & -1<m< 0,\\
      C \log N, & m=0,\\
      0 , & 0<m,
    \end{array}
    \right.
  \end{equation*}
where we have used Proposition \ref{prop..lower.bound.of.minimal.terrace.length.weak.version} again for the case $-1 < m < 0$.
\end{proof}

In preparation for the energetic consideration, we have the following result 
for the pair potential.

\begin{lem}\label{lem..e_symmetry}
  Suppose that $-1<m<n$. For $0<x<l_*$, we have $V(l_*+x)<V(l_*-x)$.
\end{lem}

\begin{proof}
  Denote $y=|l_{*}-x|^{-1}, z=|l_{*}+x|^{-1}$. Then $0<z<1=l_*<y$ and
  \begin{equation*}
    \frac{y^{n+2}-z^{n+2}}{y^{m+2}-z^{m+2}}=y^{n-m}+z^{m+2}
\left(\frac{y^{n-m}-z^{n-m}}{y^{m+2}-z^{m+2}}\right)>y^{n-m}>1>\frac{m+1}{n+1}.
  \end{equation*}
  Note that $V(l_*+x)=V(l_*-x)$ and $V'(l_*+x)=V'(l_*-x)$ at $x=0$. 
  Next we consider the following expression for $0 < x < l_*$:
  \begin{align*}
    & V''(l_*+x)-V''(l_*-x)\\
    = & -(m+1)z^{m+2}+(n+1)z^{n+2}
    +(m+1)y^{m+2}-(n+1)y^{n+2}\\
    =& \textstyle (n+1)\left\{-\frac{m+1}{n+1}\left(z^{m+2}-y^{m+2}\right)
    +(z^{n+2}-y^{n+2})\right\}\\
    =& (n+1)(z^{m+2}-y^{m+2})\left\{\frac{y^{n+2}-z^{n+2}}{y^{m+2}-z^{m+2}}-\frac{m+1}{n+1}\right\} < 0.
  \end{align*}
  Upon integrating the above inequality, the desired statement follows.
\end{proof}

\begin{lem}\label{lem..step.chain.energy.estimate}
  Suppose that $0\leq m<1$, $1<n$, and $0<\alpha\leq 1$. There exist $C$ and $N_0$ such that for any $N> N_0$ and critical point $X_N=(x_1,\cdots,x_N)^T$ of \eqref{eq..Energy.mnalphaDefinition}, for $1\leq k\leq \lfloor N^{\alpha}\rfloor-1$, we have
  \begin{equation}\label{step.chain.energy}
    \min\left(
    \sum_{i=1}^{k}V(\xi_i)\right)
    \geq
    \left\{\begin{array}{ll}
      -C\lambda_N^{-1}-C\lambda_N^{-1}(k\lambda_N)^{1-m},& 0<m<1,\\
      k\log(k\lambda_N)-C N^{\alpha}, & m=0,
    \end{array}
    \right.
  \end{equation}
  where the minimum is taken over the set 
  \begin{equation}\label{step.chain-2}
    \left\{\lambda_N\leq \xi_1<\xi_2<\cdots<\xi_k,\quad \xi_{i+1}-\xi_i\geq \lambda_N\quad\text{for}\quad i=1,2,\cdots, k-1\right\}.
  \end{equation}
  In particular, we have
  \begin{equation}
    \min\left(
    \sum_{i=1}^{k}V(\xi_i)\right)
    \geq-C\lambda_N^{-1}(N^{\alpha}\lambda_N)^{1-m},\,\,0<m<1.
  \end{equation}
\end{lem}

\begin{proof}
  Given any $\xi_i$'s from the set \eqref{step.chain-2}, we define
  $k_1=|\{i\colon 0<\xi_i\leq l_{*}\}|$ and
  $k_2=|\{i\colon l_{*}<\xi_i\}|$.
  Then $k_1+k_2=k$ and $k_1, k_2\leq \lfloor N^{\alpha}\rfloor$. 
Without loss of generality, assume $k\geq 4$ and
  $k_1,k_2\geq 1$.
  (The result remains the same if anyone of them is zero.)
  Recall that $V(\cdot)$ is monotonically decreasing on $(0,l_{*})$, and monotonically increasing on $(l_{*},\infty)$. Then
  \begin{align*}
    \sum_{i=1}^{k}V(\xi_i)
    &\geq \sum_{i=0}^{k_1-1} V(l_{*}-i \lambda_N)+ \sum_{i=0}^{k_2-1} V(l_{*}+i \lambda_N)\\
    &\geq 2V(l_{*})+\lambda_N^{-1}\int^{l_*}_{l_*-(k_1-1)\lambda_N}V(x)\,\D x+\lambda_N^{-1}\int_{l_*}^{l_*+(k_2-1)\lambda_N}V(x)\,\D x\\
    &\geq 2V(l_{*})+\lambda_N^{-1}\int_{l_*}^{l_*+(k_1-1)\lambda_N}V(x)\,\D x+\lambda_N^{-1}\int_{l_*}^{l_*+(k_2-1)\lambda_N}V(x)\,\D x\\
    &= 2V(l_{*})+\lambda_N^{-1}\Big[W(l_*+(k_1-1)\lambda_N)+W(l_*+(k_2-1)\lambda_N)-2W(l_*)\Big],
  \end{align*}
  where we have used Lemma \ref{lem..e_symmetry} in the third inequality. Since $W''(x)=V'(x)>0$, $W(x)$ is convex on $(l_*,\infty)$. By Jensen's inequality, we have 
$$W(l_*+(k_2-1)\lambda_N)+W(l_*+(k_1-1)\lambda_N)\geq 2W(l_*+\frac{k-2}{2}\lambda_N).$$ 
Then
  \begin{align*}
    \sum_{i=1}^{k}V(\xi_i)
    &\geq 2V(l_{*})+2\lambda_N^{-1}\left[W\left(l_*+\frac{k-2}{2}\lambda_N\right)-W(l_*)\right]\\
    &\geq -C\lambda_N^{-1}+2\lambda_N^{-1}W\left(l_*+\frac{k-2}{2}\lambda_N\right).
  \end{align*}

(i) If $k\lambda_N\leq 2$, then
  \begin{equation*}
    \sum_{i=1}^{k}V(\xi_i)
    \geq -C\lambda_N^{-1}\geq\left\{
      \begin{array}{ll}
        -C\lambda_N^{-1}-C\lambda_N^{-1}(k\lambda_N)^{1-m}, & 0<m<1,\\
        k\log(k\lambda_N)-CN^\alpha, & m=0,
      \end{array}
    \right.
  \end{equation*}
  where for $m=0$, we have used the facts 
$k\log(k\lambda_N)\leq k\log 2\leq 2N^\alpha$ 
and $\lambda_N N^\alpha\geq 1$ from
Proposition 
\ref{prop..lower.bound.of.minimal.terrace.length.weak.version}.

(ii) If $k\lambda_N>2$, then 
$1=l_*\leq \frac{k}{2}\lambda_N\leq l_*+\frac{k-2}{2}\lambda_N\leq k\lambda_N$.
  Note that $W(\cdot)$ is monotonically decreasing on $(l_*,\infty)$. Thus
  \begin{equation*}
    \textstyle
    W(l_*+\frac{k-2}{2}\lambda_N)
    \geq W(k\lambda_N)\geq\left\{
    \begin{array}{ll}
      -C(k\lambda_N)^{1-m}-C, & 0<m<1, \\
      k\lambda_N \log (k\lambda_N)-k\lambda_N-C, & m=0.
    \end{array}
    \right.
  \end{equation*}
  Now if $0<m<1$, then
  $
    \textstyle\sum_{i=1}^{k}V(\xi_i)
    \geq -C\lambda_N^{-1}(k\lambda_N)^{1-m}
  $ while if $m=0$, then
  $
    \textstyle\sum_{i=1}^{k}V(\xi_i)
    \geq -C\lambda_N^{-1}+2k\log(k\lambda_N)-2k
    \geq k\log(k\lambda_N)-C N^{\alpha}
  $.

Thus \eqref{step.chain.energy} is proved.
\end{proof}

\begin{lem}\label{lem..sum.of.step.chain.energy.estimate}
Let $0\leq m<1<n$ and $0<\alpha\leq 1$. There exist $C$ and $N_0$ such that for any $N> N_0$ and critical point $X_N=(x_1,\cdots,x_N)^T$ of \eqref{eq..Energy.mnalphaDefinition}, for $1\leq M\leq \lfloor N^{\alpha}\rfloor-1$, we have
  \begin{equation}
    \sum\limits_{k=1}^{M}\left[k\log(k\lambda_N)-CN^{\alpha}\right]
    \geq \frac{1}{2}M^2\log (N^\alpha\lambda_N)-CMN^\alpha.
  \end{equation}
  In particular, $\sum_{k=1}^{\lfloor N^\alpha\rfloor}\left[k\log(k\lambda_N)-CN^{\alpha}\right]\geq-CN^{2\alpha}$.
\end{lem}
\begin{proof}
  Note that $x\log x\geq -\frac{1}{e}$ for $0\leq x\leq 1$. By Proposition \ref{prop..lower.bound.of.minimal.terrace.length.weak.version}, $\log(N^\alpha\lambda_N)\geq0$. Therefore,
  \begin{align*}
     & \sum\limits_{k=1}^{M}\left[k\log(k\lambda_N)-CN^{\alpha}\right]\\
    \geq &
    -CMN^{\alpha}+\sum_{k=1}^{M}k\log(N^\alpha\lambda_N)+N^\alpha\sum_{k=1}^{M}\frac{k}{N^\alpha}\left(\log \frac{k}{N^\alpha}\right)\\
    \geq& -CMN^{\alpha}+\frac{M(M+1)}{2}\log(N^\alpha\lambda_N)-\frac{MN^\alpha}{e}\\
    \geq&\, \frac{1}{2}M^2\log (N^\alpha\lambda_N)-CMN^\alpha
  \end{align*}
proving the desired statement.
\end{proof}

\subsection{Proof of Theorem \ref{thm..Bunching} (Bunching Regime): 
Matching Bounds for $\lambda_N$ and $E_N$}
\label{sec..lower.bounds.terrace.lengths}
Using the results from the previous section, 
we establish here a lower bound for $\lambda_N$ and a connection 
between $\lambda_N$ and the minimum energy $E_N$. 
Together they lead to matching
lower and upper bounds for both quantities.

\noindent
{\bf Proof of Theorem \ref{thm..Bunching}(B) 
(lower bound for $\lambda_N$).}
  By Proposition \ref{prop..upper.bound.terrace.length}, we have $\lambda_N\leq 1<l_{**}$. Let $l_i=\lambda_N$. Utilizing the force balance condition 
\eqref{ForceBalanceCond}, we have 
\begin{align}
    0
    &= \sum\limits_{j\leq i\leq k, k-j\leq \lfloor N^{\alpha}\rfloor-1}V'(l_j+\cdots+l_i+\cdots+l_k)\nonumber\\
    &= V'(l_i)+F_i+\sum_{j=1\vee(i-\lfloor N^\alpha\rfloor+1)}^{i-1} G_j\nonumber\\
    &= \lambda_N^{-m-1}-\lambda_N^{-n-1}+F_i+\sum_{j=1\vee(i-\lfloor N^\alpha\rfloor+1)}^{i-1} G_j,\label{eq..force.balance.step.chain}
  \end{align}
  where $F_i$ is the summation of chains starting from the $i$th terrace and $G_j$ is the summation of chains crossing over the $i$th terrace:
  \begin{align*}
    F_i
    &= \sum_{k=i+1}^{(i+\lfloor N^{\alpha}\rfloor-1)\wedge(N-1)}V'(l_i+\cdots+l_k),\\
    G_j
    &= \sum_{k=i}^{(j+\lfloor N^{\alpha}\rfloor -1)\wedge(N-1)}V'(l_j+\cdots+l_k),\,\,1\vee(i-\lfloor N^\alpha\rfloor+1)\leq j\leq i-1.
  \end{align*}

The estimates of $F_i$ and $G_j$ are divided into the following cases.

  \noindent
  \textbf{Case (i):} $-1<m<0$, $1<n$.
  By Lemma \ref{lem..step.chain.force.estimate},
  we have
  $F_i\leq C\lambda_N^{-1}(N^\alpha\lambda_N)^{-m}$ and
  $
    G_j\leq C\lambda_N^{-1}(N^\alpha\lambda_N)^{-m}
  $
  for $1\vee(i-\lfloor N^\alpha\rfloor+1)\leq j\leq i-1$.
  Substituting these into \eqref{eq..force.balance.step.chain} leads to
  \begin{equation*}
    0\leq\lambda_N^{-m-1}-\lambda_N^{-n-1}+CN^\alpha\lambda_N^{-1}(N^\alpha\lambda_N)^{-m}.
  \end{equation*}
  Thus $\lambda_N^{-n-1}\leq \lambda_N^{-m-1}+CN^\alpha\lambda_N^{-1}(N^\alpha\lambda_N)^{-m}\leq CN^{(1-m)\alpha}\lambda_N^{-1-m}$.
  Therefore, 
  \begin{equation*}
      CN^{-\frac{(1-m)\alpha}{n-m}}\leq\lambda_N.
  \end{equation*}

  \noindent
  \textbf{Case (ii):} $m=0$, $1<n$.
  By Lemma \ref{lem..step.chain.force.estimate} again,
  we have
  $F_i\leq C\lambda_N^{-1}\log N$ and
  $
    G_j\leq C\lambda_N^{-1}\log N
  $
  for $1\vee(i-\lfloor N^\alpha\rfloor+1)\leq j\leq i-1$.
  Substituting these with $m=0$ into \eqref{eq..force.balance.step.chain} leads to
  \begin{equation*}
    0\leq\lambda_N^{-1}-\lambda_N^{-n-1}+CN^\alpha\lambda_N^{-1}\log N.
  \end{equation*}
  Thus $\lambda_N^{-n-1}\leq \lambda_N^{-1}+CN^\alpha\lambda_N^{-1}\log N\leq CN^\alpha\lambda_N^{-1}\log N$.
  Therefore, 
  \begin{equation*}
    CN^{-\frac{\alpha}{n}}(\log N)^{-\frac{1}{n}}\leq \lambda_N.
  \end{equation*}

  \noindent
  \textbf{Case (iii):} $0<m<1$, $1<n$.
  Without loss of generality, suppose that $i\leq \frac{N}{2}$. Otherwise, we switch the order of $\{l_k\}_{k=1}^{N-1}$ by setting $l_{i}'=l_{N-i}$ and 
analyze $l_{i}'$.
  By Lemma \ref{lem..step.chain.force.estimate}, we have
  \begin{equation}
    F_i\leq C\lambda_N^{-1}.\label{eq..force.chain.F_i}
  \end{equation}
 Let $S_i=\{j:l_j+\cdots+l_i\leq l_{**}, i-j\leq \lfloor N^\alpha\rfloor-1\}$ and $j_0=\min S_i$. Note that
$S_i\neq\emptyset$ as $i\in S_i$. Hence $j_0$ is well-defined.
  For $j$ satisfying $j_0\leq j\leq i-1$, utilizing Lemma \ref{lem..step.chain.force.estimate}, we have
  $
    G_j\leq C\lambda_N^{-1}
  $.
  Note that $|\{j_0,j_0+1,\cdots,i-1\}|< i-j_0+1\leq l_{**}\lambda_N^{-1}\leq C\lambda_N^{-1}$. Therefore
  \begin{equation}
    \sum_{j=j_0}^{i-1} G_j\leq C\lambda_N^{-2}.\label{eq..force.chain.G_j.part1}
  \end{equation}
  For $j$ satisfying $1\vee(i-\lfloor N^\alpha\rfloor+1)\leq j\leq j_0-1$, we have $l_j+\cdots+l_i>l_{**}$. Thus
  \begin{align*}
    G_j
    &\leq V'(l_j+\cdots+l_i)+\sum\limits_{k=1}^{\infty}V'(l_j+\cdots+l_i+k\lambda_N)\\
    &\leq V'(l_j+\cdots+l_i)+\lambda_N^{-1}\int_{l_j+\cdots+l_i}^{\infty}V'(x)\,\D x\\
    &= V'(l_j+\cdots+l_i)-\lambda_N^{-1}V(l_j+\cdots+l_i).
  \end{align*}
  In the last step, we have used the fact that $\lim_{x\rightarrow +\infty}V(x)=0$ due to $0<m<n$. Utilizing Lemma \ref{lem..step.chain.force.estimate} again, we have
  $$\sum_{j=1\vee(i-\lfloor N^\alpha\rfloor+1)}^{j_0-1} V'(l_j+\cdots+l_i)\leq C\lambda_N^{-1}.$$
  Notice that $W'(x)=V(x)<0$ for $x>l_{**}$. Thus $W(\cdot)$ is monotonically decreasing on $(l_{**},\infty)$ for $0<m<n$.
  Then
  \begin{align}
    \sum_{j=1\vee(i-\lfloor N^\alpha\rfloor+1)}^{j_0-1}G_j
    &\leq C\lambda_N^{-1}-\lambda_N^{-1}\sum_{j=1\vee(i-\lfloor N^\alpha\rfloor+1)}^{j_0-1}V(l_j+\cdots+l_i)\nonumber\\
    &\leq C\lambda_N^{-1}-\lambda_N^{-1}\left[V(l_{**})+\sum_{k=1}^{\lfloor N^\alpha\rfloor}V(l_{**}+k\lambda_N)\right]\nonumber\\
    &\leq C\lambda_N^{-1}-\lambda_N^{-1}\left[\lambda_N^{-1}\int_{l_{**}}^{l_{**}+N^{\alpha}\lambda_N}V(x)\,\D x\right]\nonumber\\
    &\leq C\lambda_N^{-1}-\lambda_N^{-2} W(l_{**}+N^{\alpha}\lambda_N)+\lambda_N^{-2}W(l_{**})\nonumber\\
    &\leq C\lambda_N^{-2}-\lambda_N^{-2} W((l_{**}+1)N^{\alpha}\lambda_N),\label{eq..sum.force.chain.G_j}
  \end{align}
  where in the last inequality we have used the assumption $1\leq \lambda_N^{-1}$ and the result $1\leq N^\alpha \lambda_N$ from Proposition \ref{prop..lower.bound.of.minimal.terrace.length.weak.version}. Now \eqref{eq..sum.force.chain.G_j} leads to
  \begin{equation}
    \sum_{j=1\vee(i-\lfloor N^\alpha\rfloor+1)}^{j_0-1}G_j
    \leq C\lambda_N^{-2}+C \lambda_N^{-2}(N^\alpha\lambda_N)^{1-m}
    \leq C \lambda_N^{-2}(N^\alpha\lambda_N)^{1-m}.\label{eq..force.chain.G_j.part2A}
  \end{equation}
  Substituting \eqref{eq..force.chain.F_i}, \eqref{eq..force.chain.G_j.part1}, and \eqref{eq..force.chain.G_j.part2A} into \eqref{eq..force.balance.step.chain}, we obtain
  \begin{equation*}
    0\leq \lambda_N^{-m-1}-\lambda_N^{-n-1}+C \lambda_N^{-2}(N^\alpha\lambda_N)^{1-m}.
  \end{equation*}
  Thus $\lambda_N^{-n-1}\leq \lambda_N^{-m-1}+C \lambda_N^{-2}(N^\alpha\lambda_N)^{1-m}\leq CN^{(1-m)\alpha}\lambda_N^{-1-m}$. Therefore, 
  \begin{equation*}
      CN^{-\frac{(1-m)\alpha}{n-m}}\leq \lambda_N.
  \end{equation*}
This completes the proof of Theorem \ref{thm..Bunching}(B).

\noindent
{\bf Proof of Theorem \ref{thm..Bunching}(C) 
(lower bound $w_N$).}
  This follows immediately from the lower bound in Theorem \ref{thm..Bunching}(B) by the fact that $N\lambda_N\leq w_N$.

To continue, note that there is an interesting relation between the estimates 
of the minimum energy $E_N$ and the minimal terrace length $\lambda_N$. 
Roughly speaking, 
$E_N$ is a ``monotonically increasing function''
of $\lambda_N$ (when it is sufficiently small).
This fact is then used to relate
the lower (respectively, upper) bound of $E_N$ to 
the lower (respectively, upper) bound of $\lambda_N$.

\noindent
{\bf Proof of Theorem \ref{thm..Bunching}(A)(lower bound for $E_N$) and 
(B)(upper bound for $\lambda_N$).}
  We first claim that for $-1<m<1$, $1<n$, there exist $C$ and $N_0$ such that for any $N> N_0$ and critical point $X_N=(x_1,\cdots,x_N)^T$ of \eqref{eq..Energy.mnalphaDefinition}, the following holds.
  \begin{equation}
    E[X_N]\geq
    \left\{\begin{array}{ll}
      C N^{1+(1-m)\alpha}\lambda_N^{-m}, & -1<m<0, 1<n, 0<\alpha\leq 1,\\
      N^{1+\alpha}\log(N^{\alpha}\lambda_N)-CN^{1+\alpha}, & m=0, 1<n, 0<\alpha<1,\\
      \frac{1}{2}N^2\log (N\lambda_N)-CN^2, & m=0, 1<n, \alpha=1,\\
      -CN\lambda_N^{-1}(N^\alpha\lambda_N)^{1-m}, & 0<m<1, 1<n, 0<\alpha\leq 1.
    \end{array}
    \right.\label{eq..energy.lower.bound}
  \end{equation}
Note that all the functions in the right side of
\eqref{eq..energy.lower.bound} are {\em increasing} functions of $\lambda_N$.
  Given the above, the energy upper bound for $E_N$ (from Theorem
  \ref{thm..EnergyUpperBound}) implies the desired
  upper bound for $\lambda_N$ while the lower bound for 
  $\lambda_N$ (from Theorem \ref{thm..Bunching}(B) which is just proved)
implies the desired lower bound for $E_N$. More precisely,

\noindent
(1) for $-1<m<0, 1<n, 0<\alpha\leq 1$,
\begin{align*}
&\,\,\,CN^{1+\frac{(1-m)n\alpha}{n-m}} \geq E_N\,\,\,\text{(from Theorem \ref{thm..EnergyUpperBound} (Case C))}\\
\text{and} & \,\,\,E[X_N] \geq C N^{1+(1-m)\alpha}\lambda_N^{-m}
\,\,\,\text{(from \eqref{eq..energy.lower.bound})}\\
\text{imply}&\,\,\,
\lambda_N \leq CN^{-\frac{(1-m)\alpha}{n-m}}\,\,\,
\text{(for any minimizers),}
\end{align*}
while
\begin{align*}
\lambda_N \geq CN^{-\frac{(1-m)\alpha}{n-m}}\,\,\,\text{(from Theorem \ref{thm..Bunching}(B))}
\,\,\,\,\,\,\text{implies}\,\,\,\,\,\,
E_N \geq CN^{1+\frac{(1-m)n\alpha}{n-m}};
\end{align*}
(2) for $m=0, 1<n, 0<\alpha\leq 1$,
\begin{align*}
&\,\,\, \frac{(n-1)\alpha}{n}N^{1+\alpha}\log N
\geq E_N\,\,\,\text{(from Theorem \ref{thm..EnergyUpperBound} (Case C))}\\ 
\text{and} & \,\,\,E[X_N] \geq 
      N^{1+\alpha}\log(N^{\alpha}\lambda_N)
\,\,\,\text{(from \eqref{eq..energy.lower.bound})}\\
\text{imply} & \,\,\,
\lambda_N \leq CN^{-\frac{\alpha}{n}}\,\,\,\text{(for any minimizers),}
\end{align*}
while
\begin{align*}
&\,\,\,\lambda_N \geq CN^{-\frac{\alpha}{n}}\left(\log N\right)^{-\frac{1}{n}}
\,\,\,\text{(from Theorem \ref{thm..Bunching}(B))}\\
\text{implies}&\,\,\,
E_N \geq 
\frac{(n-1)\alpha}{n}N^{1+\alpha}\log N - CN^{1+\alpha}\log\log N.
\end{align*}
The other cases: $m=0,\,\,\alpha=1$ and $0<m,\,\,0<\alpha \leq 1$ 
follow similarly.

  Thus it remains to establish \eqref{eq..energy.lower.bound}. Its proof
  is divided into several cases.

  \noindent
  \textbf{Case (i):} $-1<m<0$, $1<n$. First let
  $P = \{(i,j): 1\leq i < j \leq N\}$. Then $P=P_1\cup P_2$, where
  \begin{align*}
    \textstyle P_1&=\left\{(i,j)\in P\colon j-i\leq 
    \Big\lfloor\frac{1}{2} N^{\alpha}\Big\rfloor\right\}\\
\quad\text{and}\quad
    P_2&=\left\{(i,j)\in P\colon \Big\lfloor \frac{1}{2} N^\alpha
    \Big\rfloor<j-i\leq \lfloor N^\alpha\rfloor\right\}.
  \end{align*}
  Note that $|P_1|=C N^{1+\alpha}+O(N)$ and $|P_2|=C N^{1+\alpha}+O(N)$.
  If $(i,j)\in P_2$, then for sufficiently large $N$, we have
  $
    x_j-x_i\geq \textstyle \frac{1}{2} N^{\alpha} \lambda_N\geq l_*
  $
  by the lower bound of $\lambda_N$ from 
  Proposition \ref{prop..lower.bound.of.minimal.terrace.length.weak.version}.
  Hence for $(i,j)\in P_2$, we have
  $
    \textstyle V(x_j-x_i)\geq V(\frac{1}{2}N^{\alpha}\lambda_N)
    \geq C (N^\alpha\lambda_N)^{-m}.
  $
On the other hand, for $(i,j)\in P_1$, we have the trivial lower bound
$V(x_j-x_i)\geq \min V \geq C >0$.
  Therefore,
  \begin{align*}
    E[X_N]
    &= \sum_{(i,j)\in P_1}V(x_j-x_i)+\sum_{(i,j)\in P_2}V(x_j-x_i)\\
    &\geq C(C N^{1+\alpha}+O(N))+(C N^{1+\alpha}+O(N))C (N^\alpha\lambda_N)^{-m}\\
    &\geq C N^{1+(1-m)\alpha}\lambda_N^{-m}.
  \end{align*}
(In the above, note that the second term is much bigger
than the first due to the facts that $m<0$ and $N^\alpha\lambda_N \gg 1$ 
(Proposition \ref{prop..lower.bound.of.minimal.terrace.length.weak.version}).)

  \noindent
  \textbf{Case (ii):} $m=0$, $1<n$, and $0<\alpha<1$. We have
  \begin{align*}
    E[X_N]
    &= \left(\sum_{i=1}^{N-\lfloor N^{\alpha}\rfloor}+\sum_{i=N-\lfloor N^{\alpha}\rfloor+1}^{N-1}\right)\left[\sum_{k=i}^{(i+\lfloor N^\alpha\rfloor-1)\wedge(N-1)} V(l_i+\cdots+l_{k})\right]\\
    &\geq (N-\lfloor N^{\alpha}\rfloor) [N^{\alpha}\log (N^\alpha\lambda_N)-CN^{\alpha}]
    +\sum_{k=1}^{\lfloor N^{\alpha}\rfloor-1}[k\log (k\lambda_N)-CN^{\alpha}]\\
    &\geq N^{1+\alpha}\log(N^{\alpha}\lambda_N)-\lfloor N^\alpha\rfloor N^\alpha\log(N^\alpha \lambda_N)-CN^{1+\alpha}-CN^{2\alpha}\\
    &\geq N^{1+\alpha}\log(N^{\alpha}\lambda_N)-CN^{1+\alpha}
  \end{align*}
  where in the last inequality we have used the estimate
  $\lambda_N\leq l_*$ from Proposition 
  \ref{prop..upper.bound.terrace.length} and the fact that
  for $0 < \alpha < 1$, the terms 
  $N^{2\alpha}\log N^\alpha, N^{2\alpha}$ are both bounded by 
  $N^{1+\alpha}$ for large $N$.

  \noindent
  \textbf{Case (iii):} $m=0$, $1<n$, and $\alpha=1$. We have
  \begin{align*}
    E[X_N]
    &= \sum_{i=1}^{N-1}\left[\sum_{k=i}^{N-1} V(l_i+\cdots+l_{k})\right]
    \geq \sum_{k=1}^{N-1}[k\log (k\lambda_N)-CN]\\
    &\geq \frac{1}{2}N^2\log (N\lambda_N)-CN^2.
  \end{align*}

  \noindent
  \textbf{Case (iv):} $0<m<1$, $1<n$. We have
  \begin{align*}
    E[X_N]
    &= \sum_{i=1}^{N-1}\left[\sum_{k=i}^{(i+\lfloor N^\alpha\rfloor-1)\wedge(N-1)} V(l_i+\cdots+l_{k})\right]\\
    &\geq
-C\sum_{i=1}^{N-1}\left[
\lambda_N^{-1}(N^\alpha\lambda_N)^{1-m}
\right]\\
    &\geq -CN\lambda_N^{-1}(N^\alpha\lambda_N)^{1-m}.
  \end{align*}
Note that in the above, we have used
Lemma \ref{lem..step.chain.energy.estimate} in the first inequalities 
for Cases (ii)--(iv) to
give a lower bound for the energy of a step chain while the calculus
Lemma \ref{lem..sum.of.step.chain.energy.estimate} is only used for Cases
(ii) and (iii).

Thus \eqref{eq..energy.lower.bound} is completely verified.

\begin{rmk}
  In many applications, people use the truncated and shifted potential instead of merely truncated potential here. In other words, their pair potential reads as $\tilde{V}(x_j-x_i)=V(x_j-x_i)-V_N$, where $V_N$ is a constant which may depend on $N$. Although the minimum energy is modified with this shifted potential, the minimizer remains the same as ours.
\end{rmk}

\subsection{Completion of Proof of Theorem \ref{thm..Bunching}: 
Upper Bounds for $w_N$} \label{sec..UpperBoundSystemSize}

Before proving the statement, we note again the obvious lower bound 
$N\lambda_N \lesssim w_N$. On the other hand, 
from the scaling law for the minimum energy $E_N$, we expect that the bunch shape is
almost linear so that we should also have $w_N \lesssim N\lambda_N$. 
We will verify this in the case $-1<m\leq 0$, $1<n$, and $\alpha=1$ in which
the attraction part of the pair potential is quite strong. 
In this regime, we are able to obtain an estimate for $w_N$ which is 
``sharp'' in the sense that the exponents for the lower and upper bounds 
``almost match''. This is stronger than
the statement of Theorem \ref{thm..Bunching}(C) but is proved only
for a smaller range of the parameters.

\begin{prop}\label{thm..size.system.bunching.alpha=1.further.result}
  Let $-1<m\leq 0$, $1<n$, and $\alpha=1$. For any $\delta>0$ and $0<\delta'<\frac{1}{2}$, there exist $C_{\delta'}$ and $C_{\delta,\delta'}$ 
such that for any $N$ and minimizer $X_N=(x_1,\cdots,x_N)^T$ of 
\eqref{eq..Energy.mnalphaDefinition}, we have
  \begin{equation}\label{wN.upper.bd.sharp}
    w_{N,\delta'} \leq\left\{
    \begin{array}{ll}
      C_{\delta'}N^{\frac{n-1}{n-m}}, & -1<m<0,\\
      C_{\delta,\delta'}N^{\frac{n-1}{n}+\delta}, & m=0,
    \end{array}
    \right.
  \end{equation}
  where $w_{N,\delta'}:= x_{N-\lfloor\delta' N\rfloor}-x_{\lfloor \delta' N\rfloor}$.
\end{prop}

\begin{proof}
  Let $P=\left\{(i,j)\colon 1\leq i<j\leq N\right\}$. We partition $P$ into 
  $P_1, P_2, P_3$ where
  \begin{align*}
    P_1
    &= \left\{(i,j)\in P\colon 1\leq i\leq \lfloor \delta' N \rfloor,\,\,i<j<N-\lfloor\delta' N\rfloor)\right\},\\
    P_2
    &= \left\{(i,j)\in P\colon 1\leq i\leq \lfloor \delta' N \rfloor, \,\,N-\lfloor\delta' N\rfloor\leq j\leq N \right\},\\
    P_3
    &= \left\{(i,j)\in P\colon \lfloor \delta' N \rfloor<i<j \leq N\right\}.
  \end{align*}
  Note that
  $
  |P_1|+|P_3|=\textstyle \left(\frac{1}{2}-\delta'^2\right) N^2+O(N)
  $
  and
  $
  |P_2|=\textstyle \delta'^2 N^2+O(N)
  $.
  In addition, for $(i,j)\in P_2$, we have
  $x_j-x_i\geq w_{N,\delta'}\geq(N-2\lfloor\delta'N\rfloor)\lambda_N\gg l_*$
  for sufficiently large $N$.

  For $-1<m<0$, the monotonicity of $V(\cdot)$ on $(l_*,\infty)$ leads to
  $V(x_j-x_i)\geq V(w_{N,\delta'})\geq C w_{N,\delta'}^{-m}$ for 
$(i,j)\in P_2$ and sufficiently large $N$.
  Thus, 
  \begin{align*}
    E_N
    &= \sum_{(i,j)\in P_1\cup P_3}V(x_j-x_i)+\sum_{(i,j)\in P_2}V(x_j-x_i)\\
    &\geq \textstyle -C\left[(\frac{1}{2}-\delta'^2)N^2+O(N)\right]
+C w_{N,\delta'}^{-m}\left[\delta'^2 N^2+O(N)\right]\\
    &\geq C_{\delta'}w_{N,\delta'}^{-m}N^2.
  \end{align*}
  Recall that $E_N\leq CN^{1+\frac{(1-m)n}{n-m}}$.
  Therefore $w_{N,\delta'}\leq C_{\delta'}N^{\frac{n-1}{n-m}}$.

  If $m=0$, then $V(x_j-x_i)\geq V(w_{N,\delta'})\geq \log w_{N,\delta'}$ for 
$(i,j)\in P_2$ and sufficiently large $N$. 
By Lemmas \ref{lem..step.chain.energy.estimate} and \ref{lem..sum.of.step.chain.energy.estimate}, we have
  \begin{align*}
    E_N
    &= 
\sum_{(i,j)\in P_1}V(x_j-x_i)
+\sum_{(i,j)\in P_3}V(x_j-x_i)
+\sum_{(i,j)\in P_2}V(x_j-x_i)\\
    &= 
\sum_{i=1}^{\lfloor \delta'N \rfloor} \sum_{j: (i,j)\in P_1}V(x_j-x_i)
+ \sum_{i={\lfloor \delta'N \rfloor}+1}^{N-1} \sum_{j: (i,j)\in P_3}V(x_j-x_i)
+\sum_{(i,j)\in P_2}V(x_j-x_i)\\
    &\geq \sum_{k=N-\lfloor\delta' N\rfloor-\lfloor \delta' N\rfloor}^{N-\lfloor\delta' N\rfloor-1}\left[k\log (k\lambda_N)-CN\right]
    +\sum_{k'=1}^{N-\lfloor\delta' N\rfloor -1}\left[k'\log (k'\lambda_N)-CN\right]\\
    &~~+(\log w_{N,\delta'})[\delta'^2 N^2+O(N)]\\
    &\geq \left(\frac{1}{2}-\delta'^2\right)N^2\log (N\lambda_N)+\delta'^2N^2\log w_{N,\delta'}-C_{\delta'}N^2\\
    &\geq \frac{n-1}{n}\left(\frac{1}{2}-\delta'^2\right)N^2\log N+\delta'^2N^2\log w_{N,\delta'}-C_{\delta'}N^2\log\log N.
  \end{align*}
Note that in the first inequality above, we have used the change of
indices $k=N-\lfloor\delta'N\rfloor - i$ and 
$k'=N- i$.
  Moreover, in the last inequality we have used the fact that
$\lambda_N\geq CN^{-\frac{1}{n}}(\log N)^{-\frac{1}{n}}$.
Recall that in Case C in Section \ref{sec..heuristic.scaling.laws}, we have $E_N\leq \frac{n-1}{2n}N^2\log N+CN^2$.
Therefore, for any $\delta>0$, we have 
for sufficiently large $N$ that
$$w_{N,\delta'}\leq C_{\delta'} N^{\frac{n-1}{n}}(\log N)^{C_{\delta,\delta'}'}\leq C_{\delta,\delta'}N^{\frac{n-1}{n}+\delta}$$
completing the proof of \eqref{wN.upper.bd.sharp}.
\end{proof}

Now we prove the upper bounds in Theorem \ref{thm..Bunching}(C). 
The statements to be proved ``seem weaker'' than Proposition 
\ref{thm..size.system.bunching.alpha=1.further.result} 
as the exponents in the lower and upper bounds do not match. However
they cover a wider range of $m$ and $\alpha$ and are applicable for 
any critical point of $E$. 
The proof is based on a novel covering idea which is 
completely different from the earlier parts of this paper and also of 
\cite{Luo2016-p737-771}.

\noindent{\bf Proof of Theorem \ref{thm..Bunching}(C).}
  Choose any $\delta$ such that $0<\delta<\alpha\leq 1$. We define 
the terrace index set $I:=\{1,2,\cdots,N-1\}$ and terrace intervals 
$T_i=[x_i,x_{i+1})$ for $i\in I$. Let
  \begin{equation*}
    M:=\left\lfloor\frac{x_{N-\lfloor N^\delta\rfloor}-x_{\lfloor N^\delta\rfloor}}{l_{**}}\right\rfloor.
  \end{equation*}
  Since $N\lambda_N \gg 1$, we have $M\gg 1$ for sufficiently large $N$. 
We further define
  \begin{align}
  K_\tau & =[x_{\lfloor N^\delta\rfloor}+(\tau-1)l_{**},x_{\lfloor N^\delta\rfloor}+\tau l_{**}) & \text{for}\,\,\, \tau=1,\cdots,M,\\
  \bar{K}_\tau & = 
K_{\tau-1}\cup K_\tau\cup K_{\tau+1} & \text{for}\,\,\,\tau=2,\cdots,M-1.
\end{align}
  We call $K_\tau$ a $T$-dense (terrace-dense) interval if $\left|\{j:T_j\subset \bar{K}_\tau\}\right|\geq N^{\frac{\delta}{3}}$, and $T$-sparse (terrace-sparse) interval if $\left|\{j:T_j\subset \bar{K}_\tau\}\right|<N^{\frac{\delta}{3}}$.
  We further define dense, sparse, and boundary
terrace index set as follows:
  \begin{align}
    I_\text{D}
        &= \left\{i\in I: T_i\subset \bar{K}_\tau\quad\text{for some T-dense}\,\,K_\tau\right\},\\
    I_\text{S}
    &= \left\{i\in I\backslash{I_\text{D}}: x_i\in\bigcup_{\tau=2}^{M-1}K_\tau\right\},\\
    I_\text{B}
    &= I\backslash(I_\text{D}\cup I_\text{S}).
  \end{align}

Now we analyze the length of the intervals in the above sets. 

  {\bf 1. Estimation of $\sum_{i\in I_\text{B}}|T_i|$.}
For any $i\in I_\text{B}$, one of the following
three cases holds:
$$
i< \lfloor N^\delta\rfloor,\,\,\,
i\geq N-\lfloor N^\delta\rfloor,\,\,\,
\text{or}\,\,\,
x_i\in K_1\cup K_M\cup\left[x_{\lfloor N^\delta\rfloor}+M l_{**}, x_{N-\lfloor N^\delta\rfloor}\right).
$$
In the last case, we have
$T_i\subset K_1\cup K_2 \cup K_M\cup\left[x_{\lfloor N^\delta\rfloor}+M l_{**}, x_{N-\lfloor N^\delta\rfloor}\right)$. By
Proposition \ref{prop..upper.bound.terrace.length}, 
we have $l_i\leq l_*$ for all $i\in I$. Therefore,
  \begin{equation}
    \sum_{i\in I_\text{B}}|T_i|\leq 2N^\delta l_*+4l_{**}\leq C N^\delta.\label{eq..I_B.length}
  \end{equation}

  {\bf 2. Estimation of $\sum_{i\in I_\text{D}}|T_i|$.}
  For any $i\in I_\text{D}$, $T_i$ is covered by at most three $T$-dense $\bar{K}_\tau$. Thus
  \begin{equation}
    |I_D|\geq \frac{1}{3}\sum_{K_j\,\,T\text{-dense}}\left|\{j:T_j\subset\bar{K}_j\}\right|
    \geq \frac{1}{3}N^{\frac{\delta}{3}}|\{\tau:K_\tau\,\,T\text{-dense}\}|,\label{eq..I_D.cardinality}
  \end{equation}
  where we have used the definition of $T$-dense interval in the second inequality.
  Using the facts $|\bar{K}_\tau|=3l_{**}$, $|I_D|\leq N$, and \eqref{eq..I_D.cardinality}, the length contributed by $I_D$ can be estimated as:
  \begin{align}
    \sum_{i\in I_\text{D}}|T_i|
    &\leq\sum_{K_\tau\,\,T\text{-dense}}|\bar{K}_\tau|\nonumber\\
    &= 3l_{**}\Big|\{\tau:K_\tau\,\,T\text{-dense}\}\Big|
    \leq 9l_{**}N^{-\frac{\delta}{3}}|I_D|
    \leq CN^{1-\frac{\delta}{3}}.\label{eq..I_D.length}
  \end{align}

  {\bf 3. Estimation of $\sum_{i\in I_\text{S}}|T_i|$.}
  For any $i\in I_\text{S}$, there exists a $\tau\in\{2,\cdots, M\}$ such that $x_i\in K_\tau$.
  By the definition of the $T$-sparse interval, we deduce that
  \begin{equation*}
    |\{(j,k):j\leq i\leq k, l_j+\cdots+l_k\leq l_{**}\}|\leq N^{\frac{2\delta}{3}}.
  \end{equation*}
  Thus
  \begin{equation}
    \sum_{j\leq i\leq k, l_j+\cdots+l_k\leq l_{**}}V'(l_j+\cdots+ l_k)
    \geq N^{\frac{2\delta}{3}}V'(l_i)\geq -N^{\frac{2\delta}{3}} l_i^{-n-1}.\label{eq..less.than.l_**}
  \end{equation}

  Let $j_0=\min\{l_j:l_j+\cdots+l_i\leq l_{**}\}$ and $k_0=\max\{l_k:l_i+\cdots+l_k\leq l_{**}\}$. Obviously, $T_{j_0}, T_{k_0}\subset\bigcup_{\tau=1}^{M} K_\tau$.

  Without loss of generality, we suppose that $i\leq \frac{N}{2}$. Hence $i+\lfloor \frac{1}{3}N^\alpha\rfloor\leq N-1$ for sufficiently large $N$.
  Since $V'(l_i+\cdots+l_k)>0$ for $l_j+\cdots+l_k> l_{**}>l_*$, we have
  \begin{equation}
    \sum_{j\leq i\leq k, l_j+\cdots+l_k>l_{**}}V'(l_j+\cdots+ l_k)
    >\sum_{j=1\wedge\left(i-\lfloor \frac{1}{3}N^\alpha\rfloor\right)}^{j_0-1}
    \sum_{k=k_0+1}^{i+\lfloor \frac{1}{3}N^\alpha\rfloor}V'(l_j+\cdots+ l_k).\label{eq..more.than.l_**}
  \end{equation}
  For sufficiently large $N$, we have
  \begin{equation*}
    k_0-i+1
    \leq l_{**}\lambda_N^{-1}
    \leq \left\{
    \begin{array}{ll}
    CN^{\frac{(1-m)\alpha}{n-m}}, & -1<m<1, m\neq 0,\\
    CN^{\frac{\alpha}{n}}(\log N)^{\frac{1}{n}}, & m=0,
    \end{array}
    \right.
  \end{equation*}
  where we have used the lower bound for $\lambda_N$ established 
  in Theorem \ref{thm..Bunching}(B) and the fact that $l_{**}< C$.
  Consequently,
  $i+\lfloor\frac{1}{3}N^\alpha\rfloor-k_0\geq \lfloor\frac{1}{4}N^\alpha\rfloor+1$ for sufficiently large $N$.
  Fix any $j\in\{1\wedge\left(i-\lfloor \frac{1}{3}N^\alpha\rfloor\right),\cdots,j_0-1\}$, we have
  \begin{align}
    \sum_{k=k_0+1}^{i+\lfloor \frac{1}{3}N^\alpha\rfloor}V'(l_j+\cdots l_k)
    &\geq \sum_{k=k_0+1}^{i+\lfloor \frac{1}{3}N^\alpha\rfloor}V'(l_{**}+(k-k_0)l_{*})\nonumber\\
    &\geq \sum_{k'=1}^{\lfloor \frac{1}{4}N^\alpha\rfloor+1}V'(l_{**}+k'l_{*})\nonumber\\
    &\geq l_*^{-1}\int_{l_{**}}^{\frac{1}{4}N^\alpha l_*}V'(x)\,\D x\nonumber\\
    &\geq V\left(\frac{1}{4}N^\alpha\right)-V(l_{**}).\label{eq..fix.j.force.chain.}
  \end{align}
  Collecting \eqref{eq..less.than.l_**}, \eqref{eq..more.than.l_**}, and \eqref{eq..fix.j.force.chain.} together with the fact that
  $$
    \left|\left\{1\wedge\left(i-\lfloor \frac{1}{3}N^\alpha\rfloor\right),\cdots,j_0-1\right\}\right|\geq \lfloor N^{\delta}\rfloor+1,
  $$
  we obtain for sufficiently large $N$ that
  \begin{align*}
    \sum_{j\leq i\leq k}V'(l_j+\cdots+ l_k)
    &\geq \left(\sum_{j\leq i\leq k, l_j+\cdots+l_k\leq l_{**}}+\sum_{j\leq i\leq k, l_j+\cdots+l_k>l_{**}}\right)V'(l_j+\cdots+ l_k)\nonumber\\
    &\geq -N^{\frac{2\delta}{3}} l_i^{-n-1}+N^\delta\left[V\left(\frac{1}{4}N^\alpha\right)-V(l_{**})\right]\\
    &\geq -N^{\frac{2\delta}{3}} l_i^{-n-1}+CN^\delta.
  \end{align*}
  Thus $l_i\leq CN^{-\frac{\delta}{3(n+1)}}$.
  We now estimate the length contributed by $I_S$ as
  \begin{equation}
    \sum_{i\in I_\text{S}}|T_i|
    \leq CN^{-\frac{\delta}{3(n+1)}} |I_\text{S}|\leq CN^{1-\frac{\delta}{3(n+1)}}.\label{eq..I_S.length}
  \end{equation}
  The proof is concluded by combining 
  \eqref{eq..I_B.length}, \eqref{eq..I_D.length}, and \eqref{eq..I_S.length}.

The whole Theorem \ref{thm..Bunching} is thus proved.

Corollary
\ref{thm..BunchingCriticalPoint} follows immediately as the technique in 
the proof of the upper
bounds in Theorem \ref{thm..Bunching}(C) only makes use of the force balance
condition \eqref{ForceBalanceCond}.

\subsection{Proof of Theorem \ref{thm..Non-bunching} (Non-Bunching Regime)}
\label{sec..Proof-Non-bunching}

Now we prove Theorem \ref{thm..Non-bunching} 
which covers the non-bunching regime corresponding to the Case E
in Fig. \ref{figure..PhaseDiagram}.

\noindent
{\bf Proof of (A) (energy scaling law).} 
Suppose $\lambda_N > C$ for some constant $C$ ---
this will be proved in the next step, (B).

Now, for case (i) ($1<m<n$ and $0<\alpha\leq 1$), 
we have $V(x)<0$ for $x> x_*=(\frac{m}{n})^{\frac{1}{n-m}}$. Hence,
\begin{align*}
    E_N
    &=\sum_{i=1}^{N-1}\sum_{k=1}^{\lfloor N^\alpha\rfloor\wedge(N-i)} V(x_{i+k}-x_i)
    \geq 
    \sum_{i=1}^{N-1}\sum_{\{k\geq 1:\,\, x_{i+k}-x_i \geq x_*\}} V(x_{i+k}-x_i)\\
    &\geq (N-1)\min_{\{C_0:\,\, C_0 \geq x_*\}}
\sum_{k'=1}^{\infty} V(C_0+(k'-1)\lambda_N)\\
    &\geq -CN.
\end{align*}
Next, for case (ii) ($-1<m<n$ and $\alpha=0$), we have 
$E_N\geq (N-1)\min V\geq -|\min V|N=-CN$. 
The statements follow after combining with the upper bounds for $E_N$.

\noindent
{\bf Proof of (B) (minimal terrace length).}
  The upper bound $\lambda_N\leq C'$ is already proved in 
Proposition \ref{prop..upper.bound.terrace.length}. In particular, $C'$ can be $l_*$.
  Now we show the lower bound $C\leq \lambda_N$.

  If $0<\alpha\leq 1$, then we follow the proof of the lower bound in Theorem \ref{thm..Bunching}(B) for $0<m<1$, $1<n$, and $0<\alpha\leq 1$. We still have \eqref{eq..force.chain.F_i}, \eqref{eq..force.chain.G_j.part1}, and \eqref{eq..sum.force.chain.G_j} for $1<m<n$ and $0<\alpha\leq 1$.
  Note that $N^\alpha \lambda_N\geq 1$. Thus for $1<m<n$ and $0<\alpha\leq 1$, 
we have $W((l_{**}+1)N^{\alpha}\lambda_N)\geq-C$ for some positive constant $C$. Now from \eqref{eq..sum.force.chain.G_j}, we have
  \begin{equation}
    \sum_{j=1\vee(i-\lfloor N^\alpha\rfloor+1)}^{j_0-1}G_j
    \leq C \lambda_N^{-2}.\label{eq..force.chain.G_j.part2B}
  \end{equation}
  Substituting \eqref{eq..force.chain.F_i}, \eqref{eq..force.chain.G_j.part1}, and \eqref{eq..force.chain.G_j.part2B} into \eqref{eq..force.balance.step.chain}, we obtain
  \begin{equation*}
    0\leq \lambda_N^{-m-1}-\lambda_N^{-n-1}+C \lambda_N^{-2}
  \end{equation*}
  so that $\lambda_N^{-n-1}\leq \lambda_N^{-m-1}+C \lambda_N^{-2}\leq C\lambda_N^{-1-m}$. Hence $C\leq \lambda_N$ for some constant $C$ 
(as $m< n$).

  If $\alpha=0$, using our convention, we have
$N^\alpha=1$, i.e. the interaction is nearest neighbor.
It can easily be shown that all the critical points are {\em linear 
chains} with $l_i=l_*$ for all $i$ as $l_*$ is the only critical point of $V$. 
Then the result follows immediately. But we present 
the following argument which works more generally even for 
{\em finite range} interaction.
For this, similar to Proposition 
\ref{prop..lower.bound.of.minimal.terrace.length.weak.version}, we still
employ the force balance condition \eqref{ForceBalanceCond} leading to
  \begin{equation*}
    0\leq \lambda_N^{-m-1}-\lambda_N^{-n-1}+C.
  \end{equation*}
  Note that either $\frac{1}{2}\leq \lambda_N$ or $\lambda_N<\frac{1}{2}$. 
  For the latter case, we have 
  $$\lambda_N^{-n-1}(1-(\frac{1}{2})^{n-m})\leq \lambda^{-n-1}-\lambda^{-m-1}\leq C$$ and thence $C\leq \lambda_N$.

\noindent
{\bf Proof of (C) (system size).}
This follows immediately from $C\leq \lambda_N\leq C'$.

  Theorem \ref{thm..Non-bunching} is thus completely proved.

Again Corollary \ref{thm..NonBunchingCriticalPoint} follows immediately
as the proof of the lower bounds for $\lambda_N$ in both
Theorems \ref{thm..Bunching} and \ref{thm..Non-bunching} only
makes use the force balance condition \eqref{ForceBalanceCond}.

\appendix
\section{Proof of Theorem \ref{thm..EnergyUpperBound}: 
Upper Bounds for $E_N$}\label{ProofEnergyUpperBound}

We first state the following simple lemma without proof.
\begin{lem}\label{lem..IntegralAppro}
    (i) If $\phi(x)\geq 0$ and is monotonically decreasing on $[1,+\infty)$, then
    \begin{equation}
        \left|\sum_{k=1}^{\lfloor N^\alpha\rfloor} \phi(k)-\int_{1}^{N^\alpha}\phi(x)\,\D x\right|\leq\phi(1)
    \end{equation}
    (ii) If $\phi(x)\geq 0$ and is monotonically increasing on $[1,+\infty)$, then
    \begin{equation}
        \left|\sum_{k=1}^{\lfloor N^\alpha\rfloor} \phi(k)-\int_{1}^{N^\alpha}\phi(x)\,\D x\right|\leq\phi(N^\alpha).
    \end{equation}
\end{lem}
\begin{proof}[Proof of Theorem \ref{thm..EnergyUpperBound}.]
    We estimate the upper bound of $E[X_N^0]$ where $X_N^0=(x_1^0,\cdots,x_N^0)^T$ with $x_i^0=(i-1)l_0$ for $i=1,\cdots,N$ with some appropriate $l_0$. For convenience, we write
\begin{align}
  E[X_N^0] = e_m[X_N^0]-e_n[X_N^0]
\end{align}
where for $-1 < s$ and $s\neq 0$, we have
\begin{align}
  e_s[X_N^0]
  &= 
    \sum\limits_{1\leq i<j\leq N,\,\,j-i\leq \lfloor N^{\alpha}\rfloor}-\frac{1}{s}|x_j-x_i|^{-s} =
    -\sum_{k=1}^{\lfloor N^\alpha\rfloor}(N-k)\frac{1}{s}k^{-s}l_0^{-s}
\nonumber\\
  &= 
    -\frac{1}{s}l_0^{-s}\left(N\sum_{k=1}^{\lfloor N^\alpha\rfloor}k^{-s}-\sum_{k=1}^{\lfloor N^\alpha\rfloor}k^{1-s}\right), 
\end{align}
while for $s=0$, we have
\begin{align}
  e_s[X_N^0]
  &= 
    \sum\limits_{1\leq i<j\leq N,\,\,j-i\leq \lfloor N^{\alpha}\rfloor}\log |x_j-x_i|
    = \sum_{k=1}^{\lfloor N^\alpha\rfloor}(N-k)\log (kl_0)
  \nonumber\\
  &= 
    \sum_{k=1}^{\lfloor N^\alpha\rfloor}(N-k)\log (kl_0).
\end{align}
We will estimate the above summation by applying part (i) of 
Lemma \ref{lem..IntegralAppro} 
for $\phi(x)=x^{-s}\,(s>0)$ on $[1,+\infty)$ but part (ii) 
for 
$\phi(x)=x^{-s}\,(s<0)$, $\log x$, and $x\log x$ on $[1,+\infty)$. 

Without loss of generality, we assume $N\geq 2$.
We first give some useful upper bounds for $e_m[X^0_N]$ and $-e_n[X^0_N]$
with $m,n\neq 0$:
\begin{enumerate}
\item[(i)] for $-1<m<0$, 
\begin{align}
    e_m[X^0_N]
    &=-\frac{1}{m}l_0^{-m}\left(N\sum_{k=1}^{\lfloor N^\alpha\rfloor}k^{-m}-\sum_{k=1}^{\lfloor N^\alpha\rfloor}k^{1-m}\right)\nonumber\\
    &\leq \frac{1}{|m|}l_0^{-m}\left(N\int_1^{N^\alpha}x^{-m}\,\D x+N^{1-\alpha m}-\int_1^{N^\alpha}x^{1-m}\,\D x+N^{\alpha(1-m)}\right)\nonumber\\
    &\leq \frac{1}{|m|}l_0^{-m}\left(\frac{1}{1-m}N^{1+\alpha(1-m)}-\frac{1}{1-m}N+2N^{1-\alpha m}\right)\nonumber\\
    &\leq \frac{1}{|m|}l_0^{-m}\left(\frac{1}{1-m}N^{1+\alpha(1-m)}+2N^{1-\alpha m}\right).\label{eq..e_m,<0}
\end{align}
\item[(ii)] for $0<m<1$, 
\begin{align}
    e_m[X^0_N]
    &= \frac{1}{m}l_0^{-m}\left(\sum_{k=1}^{\lfloor N^\alpha\rfloor}k^{1-m}-N\sum_{k=1}^{\lfloor N^\alpha\rfloor}k^{-m}\right)\nonumber\\
    &\leq \frac{1}{m}l_0^{-m}\left(\int_1^{N^\alpha}x^{1-m}\,\D x+N^{\alpha(1-m)}-N\int_1^{N^\alpha}x^{-m}\,\D x+N^{1-\alpha m}\right)\nonumber\\
    &\leq \frac{1}{m}l_0^{-m}\left(\frac{1}{2-m}N^{\alpha(2-m)}-\frac{1}{1-m}N^{1+\alpha(1-m)}+\frac{1}{1-m}N+2N^{1-\alpha m}\right)\nonumber\\
    &\leq \frac{1}{m}\left(\frac{1}{2-m}-\frac{1}{1-m}\right)l_0^{-m}N^{1+\alpha(1-m)}+\frac{1}{m}\left(2+\frac{1}{1-m}\right)l_0^{-m}N\nonumber\\
    &\leq -\frac{1}{m(1-m)(2-m)}l_0^{-m}N^{1+\alpha(1-m)}+\frac{3}{m(1-m)}l_0^{-m}N.\label{eq..e_m,01}
\end{align}
\item[(iii)] for $1<m$, 
\begin{align}
    e_m[X^0_N]
    &= -\frac{1}{m}l_0^{-m}\left(N\sum_{k=1}^{\lfloor N^\alpha\rfloor}k^{-m}-\sum_{k=1}^{\lfloor N^\alpha\rfloor}k^{1-m}\right)\nonumber\\
    &\leq -\frac{1}{m}l_0^{-m}(N+2^{-m}N-N^\alpha)\nonumber\\
    &\leq -\frac{1}{m}(2l_0)^{-m}N.\label{eq..e_m,>1}
\end{align}
\item[(iv)] for $0<n<1$,
\begin{align}
    -e_n[X^0_N]
    \leq \frac{1}{n}l_0^{-n}N\sum_{k=1}^{\lfloor N^\alpha\rfloor}k^{-n}
    &\leq \frac{1}{n}l_0^{-n}N\left(\int_{1}^{N^\alpha} x^{-n}\,\D x+N^{-\alpha n}\right)\nonumber\\
    &\leq \frac{1}{n(1-n)}l_0^{-n}N^{1+\alpha(1-n)}+\frac{1}{n}l_0^{-n}N^{1-\alpha n}.\label{eq..e_n,01}
\end{align}
\item[(v)] for $1<n$,
\begin{align}
    -e_n[X^0_N]
    \leq \frac{1}{n}l_0^{-n}N\sum_{k=1}^{\lfloor N^\alpha\rfloor}k^{-n}
    &\leq \frac{1}{n}l_0^{-n}N\left(\int_{1}^{N^\alpha} x^{-n}\,\D x+1\right)\nonumber\\
    &= \frac{1}{n}l_0^{-n}N\left(1+\frac{1}{n-1}-\frac{1}{n-1}N^{-\alpha(n-1)}\right)\nonumber\\
    &\leq \frac{1}{n-1}l_0^{-n}N.\label{eq..e_n,>1}
\end{align}
\end{enumerate}

Now we proceed to prove the theorem.
We remark that the classification and computation in the 
following cases are quite tedious. 
As the goal is to obtain upper bounds, we certainly would like the bounds to be 
``as low as'' possible and preferably with negative prefactor. 
Hence in some cases, we will be very careful in choosing the constants
-- see the sub-cases A1, A2 and C4 in the following.

\noindent
{\bf Case (A): $-1 < m < n <1$, $0 < \alpha \leq 1$.} 
Overall, we will choose $l_0 \sim N^{-\alpha}$ and 
the bounds obtained are of the type $E_N \lesssim N$. The first two cases
cover the regime $mn > 0$ while the rest cover the regime $mn\leq 0$.

\noindent
{\bf (A1)} $0<m<n<1$. Collecting \eqref{eq..e_n,01} and \eqref{eq..e_m,01} together with $l_0=C_0 N^{-\alpha}$ and $C_0=\left(\frac{2m(1-m)(2-m)}{n(1-n)}\right)^{\frac{1}{n-m}}$, we obtain
\begin{align}
    E[X^0_N]
    &\leq \left(\frac{1}{n(1-n)}C_0^{-n}-\frac{1}{m(1-m)(2-m)}C_0^{-m}\right)N^{1+\alpha}\nonumber\\
    &~~+\frac{1}{n}C_0^{-n}N+\frac{3}{m(1-m)}C_0^{-m}N^{1+\alpha m}\nonumber\\
    &\leq -\frac{1}{2m(1-m)(2-m)}C_0^{-m}N^{1+\alpha}+\left(\frac{1}{n}C_0^{-n}+\frac{3}{m(1-m)}C_0^{-m}\right)N^{1+\alpha m},
\end{align}
where we used the fact that $$\frac{1}{n(1-n)}C_0^{-n}-\frac{1}{m(1-m)(2-m)}C_0^{-m}=-\frac{1}{2m(1-m)(2-m)}C_0^{-m}.$$

\noindent
{\bf (A2)} $-1<m<n<0$.  For $-1<n<0$, we have
\begin{align}
    -e_n[X^0_N]
    &= \frac{1}{|n|}l_0^{-n}\left(\sum_{k=1}^{\lfloor N^\alpha\rfloor}k^{1-n}-N\sum_{k=1}^{\lfloor N^\alpha\rfloor}k^{-n}\right)\nonumber\\
    &\leq \frac{1}{|n|}l_0^{-n}\left(\int_1^{N^\alpha}x^{1-n}\,\D x+N^{\alpha(1-n)}-N\int_1^{N^\alpha}x^{-n}\,\D x+N^{1-\alpha n}\right)\nonumber\\
    &\leq \frac{1}{|n|}l_0^{-n}\left(\frac{1}{2-n}N^{\alpha(2-n)}-\frac{1}{1-n}N^{1+\alpha(1-n)}+\frac{1}{1-n}N+2N^{1-\alpha n}\right)\nonumber\\
    &\leq \frac{1}{|n|}\left(\frac{1}{2-n}-\frac{1}{1-n}\right)l_0^{-n}N^{1+\alpha(1-n)}+\frac{1}{|n|}\left(2+\frac{1}{1-n}\right)l_0^{-n}N^{1-\alpha n}\nonumber\\
    &\leq -\frac{1}{|n|(1-n)(2-n)}l_0^{-n}N^{1+\alpha(1-n)}+\frac{3}{|n|(1-n)}l_0^{-n}N^{1-\alpha n}.\label{eq..e_n,<0}
\end{align}
Collecting \eqref{eq..e_n,<0} and \eqref{eq..e_m,<0} together with $l_0=C_0 N^{-\alpha}$ and $C_0=\left(\frac{|m|(1-m)}{2|n|(1-n)(2-n)}\right)^{\frac{1}{n-m}}$, we obtain
\begin{align}
    E[X^0_N]
    &\leq \left(\frac{1}{|m|(1-m)}C_0^{-m}-\frac{1}{|n|(1-n)(2-n)}C_0^{-n}\right)N^{1+\alpha}\nonumber\\
    &~~+\frac{2}{|m|}C_0^{-m}N+\frac{3}{|n|(1-n)}C_0^{-n}N\nonumber\\
    &\leq -\frac{1}{2|n|(1-n)(2-n)}C_0^{-n}N^{1+\alpha}+\left(\frac{2}{|m|}C_0^{-m}+\frac{3}{|n|(1-n)}C_0^{-n}\right)N,
\end{align}
where we used the fact that $$\frac{1}{|m|(1-m)}C_0^{-m}-\frac{1}{|n|(1-n)(2-n)}C_0^{-n}=-\frac{1}{2|n|(1-n)(2-n)}C_0^{-n}.$$
\noindent
{\bf (A3)} $0=m<n<1$ and $0<\alpha<1$. We set $l_0=N^{-\alpha}\leq 1$. Then
\begin{align}
    e_m[X^0_N]
    &= -\sum_{k=1}^{\lfloor N^\alpha\rfloor}k\log (k l_0)+N\sum_{k=1}^{\lfloor N^\alpha\rfloor}\log(k l_0)\nonumber\\
    &= -\sum_{k=1}^{\lfloor N^\alpha\rfloor}k\log k+N\sum_{k=1}^{\lfloor N^\alpha\rfloor}\log k+(\log l_0)\left(N\lfloor N^\alpha\rfloor-\sum_{k=1}^{\lfloor N^\alpha\rfloor} k \right)\nonumber\\
    &\leq N\int_1^{N^\alpha}\log x\,\D x+N\log N^\alpha-\alpha(\log N)\left(N\lfloor N^\alpha\rfloor-\sum_{k=1}^{\lfloor N^\alpha\rfloor} k \right)\nonumber\\
    &\leq N^{1+\alpha}\log N^\alpha+N\log N^\alpha-\alpha N^{1+\alpha}\log N+\alpha N\log N+\alpha N^{2\alpha}\log N\nonumber\\
    &= 2N\log N^\alpha+N^{2\alpha}\log N^\alpha.\label{eq..e_m,zero.caseA}
\end{align}
Collecting \eqref{eq..e_n,01} and \eqref{eq..e_m,zero.caseA} together with $l_0=N^{-\alpha}$, we obtain
\begin{align}
    E[X^0_N]
    &\leq \frac{1}{n(1-n)}N^{1+\alpha}+\frac{1}{n}N+2N\log N^\alpha+N^{2\alpha}\log N^\alpha\nonumber\\
    &\leq \frac{1}{n(1-n)}N^{1+\alpha}+\frac{4}{n}N^{1\wedge 2\alpha}\log N^\alpha.
\end{align}
{\bf (A4)} $0=m<n<1$ and $\alpha=1$. We set $l_0=N^{-1}\leq 1$. Then
\begin{align}
    e_m[X^0_N]
    &= -\sum_{k=1}^{N}k\log (k l_0)+N\sum_{k=1}^{N}\log(k l_0)\nonumber\\
    &= -\sum_{k=1}^{N}k\log k+N\sum_{k=1}^{N}\log k+(\log l_0)\left(N^2-\sum_{k=1}^{N} k \right)\nonumber\\
    &\leq -\int_1^N x\log x\,\D x+N\log N+N\int_1^{N}\log x\,\D x+N\log N-\frac{1}{2}N(N-1)\log N\nonumber\\
    &\leq -\frac{1}{2}N^2\log N+\frac{1}{4}N^2+N^2\log N+2N\log N-\frac{1}{2}N^2\log N+\frac{1}{2}N\log N\nonumber\\
    &\leq 3N\log N+\frac{1}{4}N^{2}.\label{eq..e_m,zero1.caseA}
\end{align}
Collecting \eqref{eq..e_n,01} and \eqref{eq..e_m,zero1.caseA} together with $l_0=N^{-1}$, we obtain
\begin{align}
    E[X^0_N]
    &\leq \frac{1}{n(1-n)}N^{2}+\frac{1}{n}N+3N\log N+\frac{1}{4}N^{2}\nonumber\\
    &\leq \left(\frac{1}{n(1-n)}+\frac{1}{4}\right)N^{2}+\frac{4}{n}N\log N.
\end{align}
{\bf (A5)} $-1<m<0=n$ and $0<\alpha<1$. We set $l_0=N^{-\alpha}\leq 1$. Then
\begin{align}
    -e_n[X^0_N]
    &= \sum_{k=1}^{\lfloor N^\alpha\rfloor}k\log (k l_0)-N\sum_{k=1}^{\lfloor N^\alpha\rfloor}\log(k l_0)\nonumber\\
    &= \sum_{k=1}^{\lfloor N^\alpha\rfloor}k\log k-N\sum_{k=1}^{\lfloor N^\alpha\rfloor}\log k-(\log l_0)\left(N\lfloor N^\alpha\rfloor-\sum_{k=1}^{\lfloor N^\alpha\rfloor} k \right)\nonumber\\
    &\leq \int_1^{N^\alpha}x\log x\,\D x+N^\alpha\log N^\alpha-N\int_1^{N^\alpha}\log x\,\D x+N\log N^\alpha\nonumber\\
    &~~+\alpha(\log N)\left(N\lfloor N^\alpha\rfloor-\sum_{k=1}^{\lfloor N^\alpha\rfloor} k \right)\nonumber\\
    &\leq \frac{1}{2}N^{2\alpha}\log N^\alpha+N^\alpha\log N^\alpha-N^{1+\alpha}\log N^\alpha+N\log N^\alpha+\alpha N^{1+\alpha}\log N\nonumber\\
    &= \left(\frac{1}{2}N^{2\alpha}+N^{\alpha}+N\right)\log N^\alpha.\label{eq..e_n,zero.caseA}
\end{align}
Collecting \eqref{eq..e_m,<0} and \eqref{eq..e_n,zero.caseA} together with $l_0=N^{-\alpha}$, we obtain
\begin{align}
    E[X^0_N]
    &\leq \frac{1}{|m|(1-m)}N^{1+\alpha}+\frac{2}{|m|}N+\left(\frac{1}{2}N^{2\alpha}+N^{\alpha}+N\right)\log N^\alpha\nonumber\\
    &\leq \frac{1}{|m|(1-m)}N^{1+\alpha}+\frac{5}{|m|}N^{1\wedge 2\alpha}\log N^\alpha.
\end{align}
{\bf (A6)} $-1<m<0=n$ and $\alpha=1$. We set $l_0=N^{-1}\leq 1$. Then
\begin{align}
    -e_n[X^0_N]
    &= \sum_{k=1}^{N}k\log (k l_0)-N\sum_{k=1}^{N}\log(k l_0)\nonumber\\
    &= \sum_{k=1}^{N}k\log k-N\sum_{k=1}^{N}\log k-(\log l_0)\left(N^2-\sum_{k=1}^{N} k \right)\nonumber\\
    &\leq \int_1^{N}x\log x\,\D x+N\log N-N\int_1^{N}\log x\,\D x+N\log N\nonumber\\
    &~~+(\log N)\frac{1}{2}N(N-1)\nonumber\\
    &\leq \frac{1}{2}N^{2}\log N+N\log N-N^{2}\log N+N\log N+\frac{1}{2}N^{2}\log N\nonumber\\
    &= 2N\log N.\label{eq..e_n,zero1.caseA}
\end{align}
Collecting \eqref{eq..e_m,<0} and \eqref{eq..e_n,zero1.caseA} together with $l_0=N^{-\alpha}$, we obtain
\begin{align}
    E[X^0_N]
    &\leq \frac{1}{|m|(1-m)}N^{2}+\frac{2}{|m|}N+2N\log N\nonumber\\
    &\leq \frac{1}{|m|(1-m)}N^{2}+\frac{4}{|m|}N\log N.
\end{align}
\noindent
{\bf (A7)} $-1<m<0$, $0<n<1$. Collecting \eqref{eq..e_n,01} and \eqref{eq..e_m,<0} together with $l_0=N^{-\alpha}$, we obtain
\begin{align}
    E[X^0_N]
    &\leq \left(\frac{1}{n(1-n)}+\frac{1}{|m|(1-m)}\right)N^{1+\alpha}+\left(\frac{2}{|m|}+\frac{1}{n}\right)N.
\end{align}
\noindent
{\bf Case (B): $-1 < m < n =1$.}
In this case, we set $l_0=N^{-\alpha}$. Then
\begin{align}
    -e_n[X^0_N]
    &= l_0^{-1}\left(N\sum_{k=1}^{\lfloor N^\alpha\rfloor}k^{-1}-\lfloor N^\alpha\rfloor\right)
    \leq N^{1+\alpha}\left(\int_1^{N^\alpha}x^{-1}\,\D x+1\right)\nonumber\\
    &=N^{1+\alpha}\log N^\alpha+N^{1+\alpha}.\label{eq..e_n.CaseB}
\end{align}
Collecting \eqref{eq..e_m,<0}, \eqref{eq..e_m,01}, \eqref{eq..e_m,zero.caseA}, and \eqref{eq..e_m,zero1.caseA}, we obtain
\begin{align}
  e_m[X_N^0]
  &= \left\{
  \begin{array}{ll}
    \frac{1}{|m|(1-m)}N^{1+\alpha}+\frac{2}{|m|}N, & -1<m<0,\\
    -\frac{1}{m(1-m)(2-m)}N^{1+\alpha}+\frac{3}{m(1-m)}N^{1+\alpha m}, & 0<m<1,\\
    2N\log N^\alpha+N^{2\alpha}\log N^\alpha, & m=0, 0<\alpha<1,\\
    3N\log N+\frac{1}{4}N^{2}, & m=0, \alpha=1,
  \end{array}
  \right.
\end{align}
which can be summarized as
\begin{align}
    E[X^0_N]
    &\leq CN^{1+\alpha}\log N.
\end{align}
{\bf Case (C): $-1 < m < 1 < n$.}
Overall, we will take $l_0\sim N^{-\frac{\alpha(1-m)}{n-m}}$ 
but in Case C4, it is crucial that we obtain a negative prefactor. 
Hence the choice of the constant in $l_0$ is important.\\
{\bf (C1)} $-1<m<0$.
Collecting \eqref{eq..e_n,>1} and \eqref{eq..e_m,<0} together with $l_0=N^{-\frac{\alpha(1-m)}{n-m}}$, we obtain
\begin{equation}
    E[X^0_N]\leq \left(\frac{1}{n-1}+
    \frac{1}{|m|(1-m)}\right)N^{1+\frac{n(1-m)\alpha}{n-m}}+\frac{2}{|m|}N^{1-\alpha+\frac{n(1-m)\alpha}{n-m}}.
\end{equation}
{\bf (C2)}
For $m=0$ and $0<\alpha<1$, we set $l_0=N^{-\frac{\alpha}{n}}\leq 1$. Thus
\begin{align}
    e_m[X^0_N]
    &= -\sum_{k=1}^{\lfloor N^\alpha\rfloor}k\log (k l_0)+N\sum_{k=1}^{\lfloor N^\alpha\rfloor}\log(k l_0)\nonumber\\
    &= -\sum_{k=1}^{\lfloor N^\alpha\rfloor}k\log k+N\sum_{k=1}^{\lfloor N^\alpha\rfloor}\log k+(\log l_0)\left(N\lfloor N^\alpha\rfloor-\sum_{k=1}^{\lfloor N^\alpha\rfloor} k \right)\nonumber\\
    &\leq N\int_1^{N^\alpha}\log x\,\D x+N\log N^\alpha-\frac{\alpha}{n}(\log N)\left(N\lfloor N^\alpha\rfloor-\sum_{k=1}^{\lfloor N^\alpha\rfloor} k \right)\nonumber\\
    &\leq N^{1+\alpha}\log N^\alpha+N\log N^\alpha-\frac{\alpha}{n}N^{1+\alpha}\log N+\frac{\alpha}{n}N\log N+\frac{\alpha}{n}N^{2\alpha}\log N\nonumber\\
    &= \frac{(n-1)\alpha}{n}N^{1+\alpha}\log N+2(N+N^{2\alpha})\log N.\label{eq..e_m,zero,alpha<1}
\end{align}
Collecting \eqref{eq..e_n,>1} and \eqref{eq..e_m,zero,alpha<1} together with $l_0=N^{-\frac{\alpha}{n}}$, we obtain
\begin{align}
    E[X^0_N]
    &\leq \frac{1}{n-1}N^{1+\alpha}+\frac{(n-1)\alpha}{n}N^{1+\alpha}\log N+2(N+N^{2\alpha})\log N\nonumber\\
    &\leq \frac{(n-1)\alpha}{n}N^{1+\alpha}\log N+CN^{1+\alpha}.
\end{align}
{\bf (C3)}
For $m=0$ and $\alpha=1$, we set $l_0=N^{-\frac{1}{n}}\leq 1$. Thus
\begin{align}
    e_m[X^0_N]
    &= -\sum_{k=1}^{N}k\log (k l_0)+N\sum_{k=1}^{N}\log(k l_0)\nonumber\\
    &= -\sum_{k=1}^{N}k\log k+N\sum_{k=1}^{N}\log k+(\log l_0)\left(N^2-\sum_{k=1}^{N} k \right)\nonumber\\
    &\leq -\int_1^N x\log x\,\D x+N\log N+N\int_1^{N}\log x\,\D x+N\log N-\frac{1}{2n}N(N-1)\log N\nonumber\\
    &\leq -\frac{1}{2}N^2\log N+\frac{1}{4}N^2+N^2\log N+2N\log N-\frac{1}{2n}N^2\log N+\frac{1}{2n}N\log N\nonumber\\
    &\leq \frac{n-1}{2n}N^2\log N+3N\log N+\frac{1}{4}N^{2}.\label{eq..e_m,zero,alpha=1}
\end{align}
Collecting \eqref{eq..e_n,>1} and \eqref{eq..e_m,zero,alpha=1} together with $l_0=N^{-\frac{1}{n}}$, we obtain
\begin{align}
    E[X^0_N]
    &\leq \frac{1}{n-1}N^{2}+\frac{n-1}{n}N^2\log N+3N\log N+\frac{1}{4}N^{2}\nonumber\\
    &\leq \frac{n-1}{2n}N^{2}\log N+CN^2.
\end{align}
{\bf (C4)} $0<m<1$.
Collecting \eqref{eq..e_n,>1} and \eqref{eq..e_m,01} together with $l_0=C_0 N^{-\frac{\alpha(1-m)}{n-m}}$ and $C_0=\left(\frac{2m(1-m)(2-m)}{n-1}\right)^{\frac{1}{n-m}}$, we obtain
\begin{align}
    E[X^0_N]
    &\leq
    \left(\frac{1}{n-1}C_0^{-n}-\frac{1}{m(1-m)(2-m)}C_0^{-m}\right)N^{1+\frac{n(1-m)\alpha}{n-m}}\nonumber\\
    &~~+\frac{3}{m(1-m)}C_0^{-m}
    N^{1+\frac{m(1-m)\alpha}{n-m}}\nonumber\\
    &\leq
    -\frac{C_0^{-n}}{n-1}N^{1+\frac{n(1-m)\alpha}{n-m}} +\frac{3C_0^{-m}}{m(1-m)}
    N^{1+\frac{m(1-m)\alpha}{n-m}},
\end{align}
where we used the fact that $\frac{1}{n-1}C_0^{-n}-\frac{1}{m(1-m)(2-m)}C_0^{-m}=-\frac{1}{n-1}C_0^{-n}$.\\
\noindent
{\bf Case (D): $1=m < n$.} We set $l_0 = 1$. Then
\begin{align}
    e_m[X^0_N]
    &= l_0^{-1}\left(-N\sum_{k=1}^{\lfloor N^\alpha\rfloor}k^{-1}+\lfloor N^\alpha\rfloor\right)\nonumber\\
    &\leq -N\int_1^{N^\alpha}x^{-1}\, \D x+2N
    =-N\log N^\alpha+2N.
\end{align}
By \eqref{eq..e_n,>1} with $l_0=1$, we have
\begin{equation}
    -e_n[X^0_N]
    \leq \frac{1}{n-1}N.
\end{equation}
Therefore,
\begin{align}
    E[X^0_N]
    &\leq -N\log N^\alpha+\frac{3}{n-1}N.
\end{align}
{\bf Case (E).} There are two cases depending on whether 
$\alpha>0$ or not. But in either cases, $l_0 \sim 1$.\\
\noindent
{\bf (E1)} $1<m<n$ and $0<\alpha\leq 1$.
Collecting \eqref{eq..e_n,>1} and \eqref{eq..e_m,>1} together with $l_0=(\frac{2^{m+1}m}{n-1})^{\frac{1}{n-m}}$, we obtain
\begin{align}
    E[X^0_N]
    &\leq \left(\frac{1}{n-1}l_0^{-n}-\frac{1}{m}(2l_0)^{-m}\right)N\nonumber\\
    &=-\frac{1}{n-1}l_0^{-n}N,
\end{align}
where we used the fact that $\frac{1}{n-1}l_0^{-n}-\frac{2^{-m}}{m}l_0^{-m}=-\frac{1}{n-1}l_0^{-n}$.

\noindent
{\bf (E2)} $-1 < m < n$ and $\alpha=0$. 
Let $l_0=1$. Then $E[X^0_N]=(N-1)\min V\leq CN.$

All the cases are thus considered.
\end{proof}

\section*{Acknowledgements}
  This work was partially supported by the Hong Kong Research Grants Council General Research Fund 16313316 and the Purdue Research Refresh Award.

\bibliographystyle{spmpsci} 

%\end{spacing}
\end{document}